\documentclass{amsart}
\usepackage[letterpaper,hmargin=1.5in,vmargin=1in]{geometry}
\usepackage{amsmath,amssymb,amsthm,mathtools,leftidx,graphicx,url,pinlabel}
\usepackage[all]{xy}
\usepackage{color}
\usepackage{amscd}
\usepackage{verbatim}
\usepackage[colorlinks,citecolor=red,linkcolor=black, pagebackref,hypertexnames=false, breaklinks]{hyperref}

\begin{document}

\newcommand\lprime{l}
\newcommand\llprime{l'}
\newcommand\Mand{\ \text{and}\ }
\newcommand\Mor{\ \text{or}\ }
\newcommand\Mfor{\ \text{for}\ }
\newcommand\Real{\mathbb{R}}
\newcommand\RR{\mathbb{R}}
\newcommand\im{\operatorname{Im}}
\newcommand\re{\operatorname{Re}}
\newcommand\sign{\operatorname{sign}}
\newcommand\sphere{\mathbb{S}}
\newcommand\BB{\mathbb{B}}
\newcommand\HH{\mathbb{H}}
\newcommand\CH{\Cx\HH}
\newcommand\dS{\mathrm{dS}}
\newcommand\ZZ{\mathbb{Z}}
\newcommand\codim{\operatorname{codim}}
\newcommand\Sym{\operatorname{Sym}}
\newcommand\Ann{\operatorname{Ann}}
\newcommand\End{\operatorname{End}}
\newcommand\Span{\operatorname{span}}
\newcommand\Ran{\operatorname{Ran}}
\newcommand\ep{\epsilon}
\newcommand\Cinf{\cC^\infty}
\newcommand\dCinf{\dot \cC^\infty}
\newcommand\CI{\cC^\infty}
\newcommand\dCI{\dot \cC^\infty}
\newcommand\Cx{\mathbb{C}}
\newcommand\Nat{\mathbb{N}}
\newcommand\dist{\cC^{-\infty}}
\newcommand\ddist{\dot \cC^{-\infty}}
\newcommand\pa{\partial}
\newcommand\Card{\mathrm{Card}}
\renewcommand\Box{{\square}}
\newcommand\Ell{\mathrm{Ell}}
\newcommand\WF{\mathrm{WF}}
\newcommand\WFh{\mathrm{WF}_\semi}
\newcommand\WFb{\mathrm{WF}_\bl}
\newcommand\Vf{\mathcal{V}}
\newcommand\Vb{\mathcal{V}_\bl}
\newcommand\Vsc{\mathcal{V}_\scl}
\newcommand\Vz{\mathcal{V}_0}
\newcommand\Lb{L_{\bl}}
\newcommand\Hb{H_{\bl}}
\newcommand\Hbpm{H_{\bl,+-}}
\newcommand\Hbpmpm{H_{\bl,\pm\pm}}
\newcommand\Hbpp{H_{\bl,++}}
\newcommand\Hbmm{H_{\bl,--}}
\newcommand\Ker{\mathrm{Ker}}
\newcommand\Coker{\mathrm{Coker}}
\newcommand\Range{\mathrm{Ran}}
\newcommand\Hom{\mathrm{Hom}}
\newcommand\Id{\mathrm{Id}}
\newcommand\sgn{\operatorname{sgn}}
\newcommand\ff{\mathrm{ff}}
\newcommand\tf{\mathrm{tf}}
\newcommand\esssupp{\operatorname{esssupp}}
\newcommand\supp{\operatorname{supp}}
\newcommand\vol{\mathrm{vol}}
\newcommand\Diff{\mathrm{Diff}}
\newcommand\Diffsc{\mathrm{Diff}_\scl}
\newcommand\Diffb{\mathrm{Diff}_\bl}
\newcommand\DiffbI{\mathrm{Diff}_{\bl,I}}
\newcommand\Diffz{\mathrm{Diff}_0}
\newcommand\TbC{{}^{\bl,\Cx} T}\
\newcommand\Tb{{}^{\bl} T}
\newcommand\Tbc{{}^{\bl} \overline{T}}
\newcommand\Sb{{}^{\bl} S}
\newcommand\Tsc{{}^{\scl} T}
\newcommand\Ssc{{}^{\scl} S}
\newcommand\Nb{{}^{\bl} N}
\newcommand\SNb{{}^{\bl} SN}
\newcommand\Lambdab{{}^{\bl} \Lambda}
\newcommand\zT{{}^{0} T}
\newcommand\Tz{{}^{0} T}
\newcommand\zS{{}^{0} S}
\newcommand\dom{\mathcal{D}}
\newcommand\cA{\mathcal{A}}
\newcommand\cB{\mathcal{B}}
\newcommand\cE{\mathcal{E}}
\newcommand\cG{\mathcal{G}}
\newcommand\cH{\mathcal{H}}
\newcommand\cU{\mathcal{U}}
\newcommand\cO{\mathcal{O}}
\newcommand\cF{\mathcal{F}}
\newcommand\cM{\mathcal{M}}
\newcommand\cQ{\mathcal{Q}}
\newcommand\cR{\mathcal{R}}
\newcommand\cI{\mathcal{I}}
\newcommand\cL{\mathcal{L}}
\newcommand\cK{\mathcal{K}}
\newcommand\cC{\mathcal{C}}
\newcommand\cX{\mathcal{X}}
\newcommand\cY{\mathcal{Y}}
\newcommand\cP{\mathcal{P}}
\newcommand\cS{\mathcal{S}}
\newcommand\cZ{\mathcal{Z}}
\newcommand\cW{\mathcal{W}}
\newcommand\Ptil{\tilde P}
\newcommand\ptil{\tilde p}
\newcommand\chit{\tilde \chi}
\newcommand\yt{\tilde y}
\newcommand\zetat{\tilde \zeta}
\newcommand\xit{\tilde \xi}
\newcommand\taut{{\tilde \tau}}
\newcommand\phit{{\tilde \phi}}
\newcommand\mut{{\tilde \mu}}
\newcommand\sigmat{{\tilde \sigma}}
\newcommand\sigmah{\hat\sigma}
\newcommand\zetah{\hat\zeta}
\newcommand\etah{\hat\eta}
\newcommand\loc{\mathrm{loc}}
\newcommand\compl{\mathrm{comp}}
\newcommand\reg{\mathrm{reg}}
\newcommand\bl{{\mathrm b}}
\newcommand\scl{{\mathrm{sc}}}
\newcommand{\sH}{\mathsf{H}}
\newcommand{\cte}{\digamma}
\newcommand\cl{\operatorname{cl}}
\newcommand\Div{\operatorname{div}}
\newcommand\hilbert{\mathfrak{X}}

\newcommand\Hh{H_{\semi}}

\newcommand\bM{\bar M}

\newcommand\xib{{\underline{\xi}}}
\newcommand\etab{{\underline{\eta}}}
\newcommand\zetab{{\underline{\zeta}}}

\newcommand\xibh{{\underline{\hat \xi}}}
\newcommand\etabh{{\underline{\hat \eta}}}
\newcommand\zetabh{{\underline{\hat \zeta}}}

\newcommand\rhot{{\tilde\rho}}

\newcommand\hM{\hat M}

\newcommand\Op{\operatorname{Op}}
\newcommand\Oph{\operatorname{Op_{\semi}}}

\newcommand\innr{{\mathrm{inner}}}
\newcommand\outr{{\mathrm{outer}}}
\newcommand\full{{\mathrm{full}}}
\newcommand\semi{\hbar}

\newcommand\elliptic{\mathrm{ell}}
\newcommand\even{\mathrm{even}}

\newcommand\past{\mathrm{past}}
\newcommand\future{\mathrm{future}}

\newcommand{\sfs}{\mathsf{s}}
\newcommand{\sC}{\mathsf{C}}
\newcommand{\sK}{\mathsf{K}}
\newcommand{\Ham}{\mathsf{H}}
\newcommand{\Hamb}{\mathsf{H}_\bl}

\newcommand\utilde{\tilde u}
\newcommand\wtilde{\tilde w}

\newcommand\bbb{\overline{\mathbb{R}^n}}
\newcommand\scTstar{ {}^{\text{sc}} T^*}
\newcommand\scTstarc[1]{\overline{ {}^{\text{sc}} T^* #1}}
\newcommand\sw{\ell}
\newcommand\smo{k}
\newcommand{\q}{\mathsf{q}}

\newcommand{\sfl}{\mathsf{l}}
\newcommand{\sfr}{\mathsf{l}}
\newcommand \p {\partial}
\newcommand \absv [1]{\left \lvert #1 \right \rvert }
\newcommand \wt [1]{\widetilde{#1}}
\newcommand{\lp}{\left(}
\newcommand{\rp}{\right)}
\newcommand{\la}{\langle}
\newcommand{\ra}{\rangle}
\newcommand{\norm}[2][]{\left \| #2 \right \|_{#1} }
\newcommand \lra {\longrightarrow}
\newcommand{\set}[1]{\left\{ #1 \right\} }
\newcommand \specb {spec_\bl}
\newcommand \ov {\overline{u}}
\newcommand \notg {h}
\newcommand\cXpm{\mathcal{X}_{+-}}
\newcommand\cYpm{\mathcal{Y}_{+-}}
\newcommand\cXmp{\mathcal{X}_{-+}}
\newcommand\cYmp{\mathcal{Y}_{-+}}
\newcommand\rhob{\rho_{\base}}
\newcommand\rhof{\rho_{\fiber}}
\newcommand\SymDom {\p (\Tscc^* M)}
\newcommand\Tscc{{}^{\scl} \overline{T}}
\newcommand\mmuu{\gamma}

\theoremstyle{plain}
\newtheorem{definition}{Definition}
\newcommand*{\defeq}{\mathrel{\vcenter{\baselineskip0.5ex \lineskiplimit0pt

                     \hbox{\scriptsize.}\hbox{\scriptsize.}}}%
                     =}

\setcounter{secnumdepth}{3}
\newtheorem{lemma}{Lemma}[section]
\newtheorem{prop}[lemma]{Proposition}
\newtheorem{thm}[lemma]{Theorem}
\newtheorem{cor}[lemma]{Corollary}
\newtheorem{result}[lemma]{Result}
\newtheorem*{thm*}{Theorem}
\newtheorem*{prop*}{Proposition}
\newtheorem*{cor*}{Corollary}
\newtheorem*{conj*}{Conjecture}
\numberwithin{equation}{section}
\theoremstyle{remark}
\newtheorem{rem}[lemma]{Remark}
\newtheorem*{rem*}{Remark}
\theoremstyle{definition}
\newtheorem{Def}[lemma]{Definition}
\newtheorem*{Def*}{Definition}
\newcommand{\fiber}{\mathrm{fib}}
\newcommand{\base}{\mathrm{base}}

\newcommand{\mar}[1]{{\marginpar{\sffamily{\scriptsize #1}}}}
\newcommand\ah[1]{\mar{AH:#1}}
\newcommand{\jgr}[1]{{\mar{JGR:#1}}}
\newcommand{\js}[1]{{\mar{JS:#1}}}
\newcommand{\jy}[1]{{\mar{JY:#1}}}

\newcommand\ran{\mathop{\rm ran}}
\newcommand{\R}{\mathbb{R}}

\newcommand\ubar{\overline{u}}
\newcommand\ang[1]{\langle #1 \rangle}
\renewcommand\Im{\operatorname{Im}}
\renewcommand\Re{\operatorname{Re}}
\newcommand\Rnbar{\overline{\mathbb{R}^n}}

\renewcommand{\theenumi}{\roman{enumi}}
\renewcommand{\labelenumi}{(\theenumi)}

\title{Existence and asymptotics of nonlinear Helmholtz eigenfunctions}
\author{Jesse Gell-Redman}
\address{School of Mathematics and Statistics, University of Melbourne}
\email{jgell@unimelb.edu.au}
\author{Andrew Hassell}
\address{Mathematical Sciences Institute, Australian National University}
\email{Andrew.Hassell@anu.edu.au}
\author{Jacob Shapiro}
\address{Mathematical Sciences Institute, Australian National
  University}
\email{Jacob.Shapiro@anu.edu.au}
\author{Junyong Zhang}
\address{Department of Mathematics,  Beijing Institute of Technology and \\ Cardiff University, UK}
\email{zhang\_junyong@bit.edu.cn; ZhangJ107@cardiff.ac.uk}



\thanks{We thank Andr\'{a}s Vasy for helpful conversations. 
This work was supported by the Australian Research Council through grant DP180100589. J. Shapiro was supported by an AMS-Simons travel grant. J. Zhang was supported by NSFC Grants (11771041, 11831004) and H2020-MSCA-IF-2017(790623). }

\begin{abstract}
We prove the existence and asymptotic expansion of a large class of solutions to nonlinear Helmholtz equations of the form 
\begin{equation*}
(\Delta - \lambda^2) u = N[u], 
\end{equation*}
where $\Delta = -\sum_j \partial^2_j$ is the Laplacian on $\RR^n$ with
sign convention that it is positive as an operator, $\lambda$ is a
positive real number, and $N[u]$ is a nonlinear operator that is a
sum of monomials of degree $\geq p$ in $u$, $\ubar$ and their derivatives of order up to
two, for some $p \geq 2$. Nonlinear Helmholtz eigenfunctions with  $N[u]= \pm |u|^{p-1} u$ were first considered by Guti\'errez \cite{Gut}.  Such equations are of interest in part because, for certain nonlinearities $N[u]$, they furnish standing waves for nonlinear
evolution equations, that is, solutions that are time-harmonic. 

We show that, under the condition $(p-1)(n-1)/2 > 2$ and $\smo > (n-1)/2$, for every $f \in H^{\smo+2}(\mathbb{S}^{n-1})$ of sufficiently small norm, there is a nonlinear Helmholtz function taking the form 
\begin{equation*} 
u(r, \omega) = r^{-(n-1)/2} \Big( e^{-i\lambda r} f(\omega) + e^{+i\lambda r}
g(\omega) + O(r^{-\epsilon}) \Big), \text{ as } r \to \infty, \quad \epsilon > 0, 
\end{equation*}
for some $g \in H^{\smo}(\mathbb{S}^{n-1})$. Moreover, we prove the result in the general setting of asymptotically conic manifolds. The proof uses an elaboration of anisotropic Sobolev spaces $\mathcal{X}^{s, \sfr_\pm}$, $\mathcal{Y}^{s, \sfr_{\pm}}$, defined by Vasy \cite{va18},  between which the Helmholtz operator $\Delta - \lambda^2$ acts invertibly. These spaces have a variable spatial weight $\sfr_\pm$, varying in phase space and distinguising between the two `radial sets' corresponding to incoming oscillations, $e^{-i\lambda r}$, and outgoing oscillations $e^{+i\lambda r}$.  
Our spaces have, in addition, module regularity with respect to two different `test modules', and have  algebra (or pointwise multiplication) properties which allow us to treat nonlinearities $N[u]$ of the form specified above. 
\end{abstract}

\maketitle

\section{Introduction} 
In this article we prove the existence and asymptotic expansion of a large class of solutions to nonlinear Helmholtz equations of the form 
\begin{equation}\label{nonlinear-efn}
(\Delta - \lambda^2) u = N[u], 
\end{equation}
where $\Delta = -\sum_j \partial^2_j$ is the Laplacian on $\RR^n$ with the 
sign convention that it is positive as an operator, $\lambda$ is a
positive real number, and $N[u]$ is a nonlinear operator that is a
monomial in $u$, $\ubar$ and their derivatives of order up to
two. Such equations are of interest in part because, for certain nonlinearities $N[u]$, they furnish standing waves for nonlinear
evolution equations, that is, solutions that are time-harmonic. Indeed
this is the case whenever $N[e^{i\theta} u] = e^{i\theta} N[u]$, for
all $\theta \in \RR$. For example, if $N[u] = \alpha |u|^{2q} u$, then
$\Psi(z, t) = u(z) e^{i\lambda^2 t}$ solves the nonlinear
Schr\"odinger equation
\begin{equation}\label{NLS}
-i \partial_t \Psi = \Delta \Psi - \alpha |\Psi|^{2q} \Psi, 
\end{equation}
while if $N[u] = |\nabla u|^2 u$, then $v(z, t) = u(z) e^{i\lambda t}$ solves the nonlinear wave equation
\begin{equation}
(\partial_t^2  + \Delta) v = |\nabla v|^2 v. 
\end{equation}

%
%
%
%
%

In this article, we will study the existence and asymptotic behaviour of `small' solutions to 
equation \eqref{nonlinear-efn}. Moreover, we shall do this not just for the standard Laplacian on $\RR^n$ but for potential and/or metric perturbations of this Laplacian, and even more generally for the Laplacian on asymptotically conic manifolds. However, in this introduction we shall mostly discuss the flat Euclidean case, as our results are new even in this setting. 

Since the linearization of this equation at $u=0$ is just the standard Helmholtz equation, 
\begin{equation}\label{Helmholtz}
(\Delta - \lambda^2) u = 0,
\end{equation}
it is intuitively clear that `small' nonlinear eigenfunctions should
behave similarly to linear Helmholtz eigenfunctions. The structure
of these is well known.   The
space of Helmholtz eigenfunctions of polynomial growth is parametrized
by distributions on the `sphere at infinity', $\mathbb{S}^{n-1}$.  We shall only consider those eigenfunctions associated to smooth functions on the sphere at infinity. Given $f \in C^\infty(\mathbb{S}^{n-1})$, there is a unique Helmholtz eigenfunction satisfying (in standard polar coordinates, $r = |z|$, $\omega = z/|z|$)
\begin{equation}
(\Delta - \lambda^2) u_0 = 0, \quad u_0 = r^{-(n-1)/2} \Big( e^{-i\lambda r} f(\omega) + e^{+i\lambda r} g_0(\omega) + O(r^{-1}) \Big), \text{ as } \ r \to \infty, 
\label{linear-efn}\end{equation}
where $g_0 \in C^\infty(\mathbb{S}^{n-1})$ is determined by $f$. (In fact, in
the simple case of the flat Laplacian, $g_0(\omega) = i^{(n-1)/2}
f(-\omega)$, but in the presence of metric or potential perturbations,
$g_0$ is not so explicit, and is indeed related to the
scattering matrix of the perturbed operator.)  We call $f$ the
`incoming  data' or `incoming radiation pattern' for the eigenfunction
$u$, while $g_0$ is referred to as the `outgoing data' or `outgoing
radiation pattern'. It is an arbitrary choice whether to parametrize
eigenfunctions by their incoming or their outgoing data; each
determines the other. 

\subsection{Main results}

Our main result, at least as it applies to the flat Laplacian on
$\RR^n$, is that `small' nonlinear eigenfunctions can be parametrized
in a similar way.  We state our result first for the equation
\begin{equation}\label{nonlinear-efn-power}
(\Delta - \lambda^2) u = \alpha |u|^{2q} u.
\end{equation}

\begin{thm}[Main Theorem, Euclidean case]\label{thm:main1}
Let $q \in \mathbb{N}$, $p = 2q + 1$, and assume that
\begin{equation}\label{pn}
(p-1) \frac{n-1}{2} > 2. 
\end{equation}
Let $\smo$ be an integer greater than $(n-1)/2$. 
There exist  $\epsilon, \epsilon' > 0$
sufficiently small, such that 
for every $f \in H^{\smo+2}(\mathbb{S}^{n-1})$ with $\| f
\|_{H^{\smo+2}(\mathbb{S}^{n-1})} < \epsilon$, there is a solution $u$ to equation \eqref{nonlinear-efn-power}, satisfying
\begin{equation} \label{u-form}
u = r^{-(n-1)/2} \Big( e^{-i\lambda r} f(\omega) + e^{+i\lambda r}
g(\omega) + O(r^{-\epsilon'}) \Big), \text{ as } r \to \infty,
\end{equation}
for some $g \in H^{\smo}(\mathbb{S}^{n-1})$. 

Moreover, uniqueness holds in the
following sense. Fix a $C^\infty$ function $\chi(r)$ equal to zero for $r$ small and $1$ for $r$ large, and let $\sw = -1/2- \delta$ for any $\delta$ satisfying $0 < \delta \leq (4p)^{-1}$. 
Let $u_- = \chi(r) r^{-(n-1)/2} e^{-i \lambda r} f(\omega)$. Then given $f$ with $\| f
\|_{H^{\smo+2}(\mathbb{S}^{n-1})}$ sufficiently small,   there is exactly one nonlinear
eigenfunction $u$ of the form \eqref{u-form},  with the property that $u - u_-$ has small norm in the Hilbert
space $H^{2, \sw; 1, \smo}_+$  defined in \eqref{eq:double-module-sob}.  \end{thm}

\begin{rem}
The solution is a scattering type solution, not a $L^2$ solution. From the Pohozaev identity, it is known that the sign of $\alpha$ plays an important role in the existence of finite energy solution to \eqref{nonlinear-efn-power},
while the sign of $\alpha$ plays no role in Theorem \ref{thm:main1}.
\end{rem}

\begin{rem}
As mentioned above, $\Psi(z, t) = u(z) e^{i\lambda^2 t}$ is a global-in-time solution which solves \eqref{NLS} but it is time-periodic without any decay.
This is quite different from the classical finite-energy solution to \eqref{NLS}.
\end{rem}

Our proof of Theorem~\ref{thm:main1} principally makes use of the
asymptotically conic structure of $\mathbb{R}^n$ near infinity; in 
particular it uses neither the translation symmetries of
$\mathbb{R}^n$ nor exact formulae for resolvent kernels.  The more
general version of our main result is valid in the setting
of asymptotically conic manifolds.
To prepare for the definition of such spaces, let us recall that, given a compact Riemannian manifold $(\boldsymbol N, g_N)$, the metric cone over $\boldsymbol N$ is the Riemannian manifold $\boldsymbol N \times (0, \infty)_r$ with metric of the form $dr^2 + r^2 g_N$. This space is incomplete as $r \to 0$; it can be completed topologically by adding a single point at $r=0$ (the `cone point'), but it is usually not a manifold then; and if it is, the metric is usually singular at $r=0$. The sole exception is when $\boldsymbol N$ is the sphere $\mathbb{S}^{n-1}$ with its standard metric, in which case the cone over $(\boldsymbol N,g_N)$ is $\RR^n$ minus the origin, and adding the cone point recovers the missing point. 

We define an \emph{asymptotically conic manifold} to be the interior
$M^\circ$ of a compact manifold with boundary $M$, with Riemannian
metric $g$ taking a particular form near the boundary. To specify
this, let $x$ be a boundary defining function for $\partial M$ (that
is, the boundary $\partial M$ is given by $x=0$, where $x$ vanishes to
first order at $\partial M$ and $x > 0$ on $M^\circ$) and let $y =
(y_1, \dots, y_{n-1})$ be local coordinates on $\partial M$ extended
to a collar neighbourhood $\{ x \leq c \}$ of the
boundary, where $c > 0$ is some small fixed positive number. 
We assume that the metric $g$ has the property that, near any point on $\partial M$, there are coordinates $(x, y_1, \dots, y_{n-1})$ as above such that, in  this coordinate patch, $g$ takes the
form
\begin{equation}\label{ac-metric}
g = \frac{dx^2}{x^4} + \frac{h(x, y, dy)}{x^2} ,
\end{equation}
where $h$ is a smooth $(0,2)$-tensor that restricts to a metric on $\partial M$. 
This definition is better understood by passing to the variable $r = 1/x$, which goes to infinity at the boundary of $M$. The metric then takes the form 
\begin{equation}\label{ac-metric2}
g = dr^2 + r^2 h(\frac1{r}, y, dy). 
\end{equation}
If $h$ is independent of $x$ for small $x$, then this is precisely a
conic metric for large $r$, where $k = h(0, y, dy)$.  More generally,
it is asymptotic to this conic metric (smoothness of $h$ in $x$ is
equivalent to having an asymptotic expansion in powers of $1/r$ as $r
\to \infty$.)  In particular, the metric is always
complete, as the
boundary is at $r = \infty$ which is infinite distance from any
interior point. Thus, we can think of an asymptotically conic manifold
as a complete noncompact Riemannian manifold that is asymptotic, at
infinity, to the `large end of a cone', but having no conic
singularity (as a true cone usually does at $r=0$). Such spaces have
curvature tending to zero at infinity, and local injectivity radius
tending to infinity, so balls of a fixed size are asymptotically
Euclidean as their centre tends to infinity. For this reason, they are
sometimes called `asymptotically Euclidean spaces', although
`asymptotically locally Euclidean' would perhaps be a better term, as
the global structure of `infinity', that is, the boundary $\partial M$
of the compactification $M$, can be quite different from $\mathbb{S}^{n-1}$.

Particular instances of asymptotically conic manifolds include flat Euclidean space, or any compact metric perturbation of the flat metric on Euclidean space. In this case, $M$ is the radial compactification of $\RR^n$, given by the union of $\RR^n$ with the `sphere at infinity', $\mathbb{S}^{n-1}$. Connected sums of such manifolds are also asymptotically conic. The topology and geodesic dynamics on such manifolds can be intricate. For example, any convex co-compact hyperbolic manifold can have its metric modified near infinity to be asymptotically conic; while this is an artificial construction, it provides a very large class of asymptotically conic spaces with complicated topology and hyperbolic trapped set. 

\begin{thm}[Main Theorem, asymptotically conic case]\label{thm:main2}
Let $(M^\circ, g)$ be an asymptotically conic manifold of dimension
$n$, and let $V$ be a smooth function on $M$ vanishing to second order
at $\partial M$ (that is, in the `noncompact' picture, $V$ is
$O(r^{-2})$ as $r \to \infty$ with an expansion at infinity in
negative powers of $r$, obeying symbolic derivative estimates). Let $H
= \Delta_g + V$ where $\Delta_g$ is the Laplace-Beltrami operator on
$(M^\circ, g)$. Let $N(u, \ubar, \nabla u, \nabla \ubar, \nabla^{(2)}
u, \nabla^{(2)} \ubar)$ be a sum of monomial terms, each of which has
degree not less than $p$ in
$u$ and $\ubar$ and their derivatives up to order two, with
coefficients smooth on $M$, and assume that $p$ satisfies \eqref{pn}.
Let $\smo$ be an integer greater than $(n-1)/2$. There exist $\epsilon, \epsilon' > 0$
sufficiently small, such that 
for every $f \in H^{\smo+2}(\partial M)$ with $\| f
\|_{H^{\smo+2}(\partial M)} < \epsilon$,  there is a function $u$ on $M^\circ$ satisfying 
\begin{equation*}
(H - \lambda^2) u = N(u, \ubar, \nabla u, \nabla \ubar, \nabla^{(2)} u, \nabla^{(2)} \ubar)
\end{equation*}
with asymptotics 
\begin{equation}
u = r^{-(n-1)/2} \Big( e^{-i\lambda r} f(\omega) + e^{+i\lambda r} g(\omega) + O(r^{-\epsilon'}) \Big), \text{ as } r \to \infty, 
\label{nonlinear-efn-main2}\end{equation}
for some $g \in H^{\smo}(\partial M)$. Moreover, uniqueness holds in same sense as in Theorem~\ref{thm:main1}. \end{thm}

\begin{rem}\label{rem:clarify} We first clarify the meaning of a ``monomial of degree not less than $p$ in $u$ and $\ubar$ and their derivatives up to order two, with coefficients smooth on $M$''. 
These derivatives are understood to be taken with respect to a frame of vector fields that are uniformly bounded with respect to the metric $g$. Thus, as $r \to \infty$ we could take $\partial_r$ and $r^{-1} \partial_{y_j}$, for example; these are the natural analogues of the gradient in the Euclidean sense, written with respect to polar coordinates. For example, if $p=3$, then on Euclidean $\RR^n$ the nonlinear term $N$ could take the form $$|u|^2 u  + |\nabla u|^2 u + |\nabla^{(2)} u|^2 u + \frac{\partial^2 u}{\partial z_1^2} \frac{\partial u}{\partial z_2} \ubar^2 + u^5.$$ 
\end{rem}

\begin{rem}The first result along the lines of Theorem~\ref{thm:main1} was obtained by Guti\'errez \cite{Gut}. Curiously, the set of pairs $(n, p)$ treated in that paper is almost disjoint to ours: it covers the case $n=3, 4$ and $p=3$ for example, but higher $n$ and $p$ are excluded, while our method works most easily with large $n$ and $p$. In fact, in view of the condition \eqref{pn} in our two theorems we can treat $p \geq 6$ when $n=2$, $p \geq 4$ when $n=3$, $p \geq 3$ when $n =4, 5$ and $p \geq 2$ for $n \geq 6$.  We discuss previous literature more fully below. 
\end{rem}

\subsection{Strategy of the proof}\label{subsec:strategy}

The basic strategy of our proof of Theorem~\ref{thm:main1} is  fixed point argument which is similar to \cite{Gut}. Given incoming data $f$, Guti\'errez formed the linear eigenfunction $u_0$ and showed that the map
\begin{equation}\label{Phi1}
\Phi : u \mapsto u_0 + (\Delta - (\lambda + i0)^2)^{-1} \alpha |u|^{p-1} u 
\end{equation}
is a contraction map on some Banach space, provided that the norm of $u$ is sufficiently small. 
Guti\'errez used $L^q$ spaces, for example $L^4$ when $p=3$ and $n = 3, 4$. Given $u \in L^4$, it is clear that the cubic term $|u|^2 u$ lies in $L^{4/3}$, while uniform resolvent bounds of Kenig-Ruiz-Sogge \cite{KRS} and the extension restriction estimates of Stein-Tomas \cite{Tomas} are used to show that the outgoing resolvent maps $L^{4/3}$ back to $L^4$. The fixed point of $\Phi$ is a nonlinear eigenfunction, as one sees by applying $\Delta - \lambda^2$ to both sides, and it has the same incoming data as $u_0$. 

In our approach, we use polynomially weighted $L^2$-based Sobolev spaces, with an aniso\-tropic weight. 
Vasy
\cite{VD2013} has shown how to construct two families of Hilbert
spaces  between which $\Delta - \lambda^2$ maps \textit{as a bounded
  invertible operator}: 
\begin{equation}
    \Delta - \lambda^2 \colon \mathcal{X}^{s,\mathsf{l}_\pm} \lra
    \mathcal{Y}^{s-2,\mathsf{l}_\pm+1}.\label{eq:first-invertible}
\end{equation}

In \eqref{eq:first-invertible}, the space $\mathcal{Y}^{s,\mathsf{l}_\pm} = H^{s,\mathsf{l}_\pm}$ is a
variable order $L^2$-based Sobolev space. The index $s \in \RR$ is a regularity parameter,  specifying how many derivatives are locally in $L^2$, while $\sfl_\pm$ is a variable spatial weight, which varies `microlocally', i.e.\ in phase space $T^* \RR^n$.  The weight $\sfr_+$ is chosen  
so that $u \in \mathcal{X}^{s, \sfr_+}$, localized in frequency close to the incoming radial oscillation $e^{-i\lambda r}$, decays at least as $r^{-(n-1)/2
- \delta}$ with $\delta > 0$ fixed but small, while near the outgoing radial oscillation
$e^{ i \lambda r}$, slower decay, as $r^{-(n-1)/2 +
  \delta}$, is permitted. The weight $\sfr_-$ has the opposite property: the decay must be faster than $r^{-(n-1)/2}$ near the outgoing radial oscillation, but can be slower near the incoming radial oscillation. 

 This means that, for the $+$ sign, the `outgoing' expansion at infinity typical of generalized eigenfunctions is permitted, while the `incoming' expansion is not, while for the $-$ sign, the situation is reversed. This is consistent with the statement that the inverse
map to \eqref{eq:first-invertible} is, for the $+$ sign, the outgoing
resolvent $(\Delta -
 (\lambda + i0)^2)^{-1}$, and for the $-$ sign, the incoming resolvent $(\Delta -
 (\lambda - i0)^2)^{-1}$, meaning that solutions $(\Delta - \lambda^2)u = f \in
C^\infty_{c}(\mathbb{R}^n)$ with $u \in \mathcal{X}^{s,\sfr_\pm}$
admit asymptotic expansions of the form 
$$
u = r^{-(n-1)/2} e^{\pm i \lambda r} \sum_{j = 0}^\infty r^{-j} v_j,
\qquad  v_j \in C^\infty(\mathbb{S}^{n-1}).
$$ 

  The domain of \eqref{eq:first-invertible} is defined by an the priori regularity condition
\begin{equation}
  \label{eq:4}
  \mathcal{X}^{s,\sfr_\pm} := \large\{ u \in H^{s,\sfr_\pm} :  (\Delta - \lambda^2)
  u \in H^{s - 2, \sfr_\pm + 1} \large\}.
\end{equation}
The exponents $(s-2, \sfr_\pm+1)$ reflect the order $(2,0)$ of the operator $P$, as well as the ellipticity of $P$ at fibre-infinity and the fact that $P$ is of real principal type at spatial-infinity, leading to a loss of one order of decay in the spatial regularity $\sfr_\pm$. 
 It is a tautology that $\Delta - \lambda^2$ is a bounded operator
 from $\mathcal{X}^{s,\sfr_\pm}$ to
 $\mathcal{Y}^{s-2,\sfr_\pm+1}$. What is \emph{not} obvious is that
 this is an invertible map, a result due to Vasy \cite{va18} with
 methods going back to Melrose \cite{RBMSpec}, and which we give a
 detailed proof of
 below.  The inverse operator  depends on the
 choice of sign $\pm$ (the choice giving either the incoming or
 outgoing resolvent), and as in the work of Guti\'errez, although a
 choice must be made, the only effect of this choice is to determine
 whether one prescribes the incoming data $f$ in the main theorems, or
 the outgoing data $g$.

 One thus obtains an inverse mapping $R(\lambda + i 0) \colon H^{s - 2, \sfr +
  1}  \lra H^{s, \sfr}$, but this is not enough to solve nonlinear
problems. Given that our nonlinear term is assumed to be polynomial, we need to work with spaces of functions with good algebra (or multiplicative) properties. The spaces $ \mathcal{X}^{s,\sfr_\pm}$ and $\mathcal{Y}^{s,\sfr} = H^{s,\sfr_\pm}$ are not suitable for this purpose, even (surprisingly) for large $s$. Recall that,  for $s > n/2$, $H^{s, 0}$, the standard
Sobolev space of order $s$, forms an algebra, i.e.\ $H^{s, 0} \cdot
H^{s, 0}  \subset H^{s, 0}$. If we include spatial weights, then (at least for constant weights), these combine additively, in the sense that we have  for $r_1, r_2 \in
\mathbb{R}$, $H^{s, r_1} \cdot
H^{s, r_2}  \subset H^{s, r_1 + r_2}$. However, our weights are
typically negative -- indeed, they are forced to be so to obtain
bijectivity of $\Delta - \lambda^2$ -- so this will not lead to a
mapping $\Phi$ on a fixed space, as in \eqref{Phi1}; indeed, the
nonlinear operation must \emph{gain} one order of spatial decay to
account for the loss of one order in the action of the resolvent
inverting \eqref{eq:first-invertible}. 

To do this, we work with spaces with additional regularity with
respect to the differential operators with coefficients that grow
linearly at infinity but which annihilate the outgoing oscillation
$e^{i \lambda r}$. 
These are generated by the operators $r (\partial_r - i \lambda)$ and purely angular differential operators $\partial_{\omega}$. This type of regularity condition is precisely
the ``module regularity'' introduced by the second author together
with Melrose and Vasy in \cite{HMV2004}, and used by the first author with Haber and Vasy in \cite{GHV} to solve a semilinear wave equation.  Thus for $s, \sw \in \mathbb{R}, \kappa, k \in
\mathbb{N}_0$ with $\sw < -1/2, \kappa \ge 1$, we define $L^2$-based Sobolev
spaces $ H^{s, \sw; \kappa, k }_+$ in which $s$ is the order of differentiability in the usual sense, i.e.\ relative to constant coefficient vector fields,  $\sw$ is the decay rate relative to
$L^2$, $\kappa$ is the order of ``module'' differentiability just
described, and $k$ is the order of differentiability tangential in
the ``angular'' direction. (Provided $\kappa \geq 1$, we can take the spatial weight $\sw$ here to be a \emph{constant} slightly less than $-1/2$, as the module regularity itself --- which is asymmetric with respect to the incoming and outgoing oscillations, $e^{\pm i \lambda r}$ ---  enforces additional vanishing of the incoming oscillations.)  We arrive at a refinement of the mapping
property  \eqref{eq:first-invertible}, namely we obtain an invertible map 
\begin{equation}
  \label{eq:Fred-map}
  \Delta - \lambda^2 \colon \mathcal{X}_+^{s,\sw; \kappa, \smo} \lra \mathcal{Y}_+^{s-2,
    \sw+1; \kappa, \smo},
\end{equation}
where $\mathcal{Y}_+^{s,\sw;\kappa, \smo} = H_+^{s, \sw; \kappa, \smo}$ and, analogously to \eqref{eq:4}, 
the $\mathcal{X}_+^{s,\sw; \kappa, \smo}$ are given by 
\begin{equation*}
  \mathcal{X}_+^{s,\sw; \kappa, \smo} \defeq \large\{ u \in H_+^{s, \sw ; \kappa, \smo} :  (\Delta - \lambda^2)
  u \in H_+^{s - 2, \sw + 1; \kappa, \smo} \large\}. 
\end{equation*}
 See Theorem \ref{thm:Fred prop} below.  The inverse
map to \eqref{eq:Fred-map} we continue to denote by $R(\lambda + i
0)$, as it is just the restriction of the inverse of
\eqref{eq:first-invertible} to $ \mathcal{Y}_+^{s-2,
    \sw+1; \kappa, \smo}$ within an appropriate choice of
  $\mathcal{Y}^{s-2,\sfr_+ +1}$.  
For $\kappa \geq 1$ and 
$k \geq (n-1)/2$, these spaces satisfy \emph{improved} multiplicative properties. 
For example, we have 
$$
\Big( H_+^{s, \sw; \kappa, \smo} \Big)^p \subset H_+^{s, p\sw + (p-1)n/2 - \kappa; \kappa, \smo};
$$
when $\kappa=1$, we gain $(p-1)n/2 - 1$ in the spatial weight, which is
crucial as it allows us to gain the order of spatial decay discussed
above, essential to remain in a fixed space $H_+^{s, \sw; 1, \smo}$ from
the combination of  applying the nonlinear operator followed by the
resolvent.  With $\sw=\mathsf{l}_+ = -1/2 - \delta$, where $\delta$ can be taken
arbitrarily small, this leads to the condition \eqref{pn}.

For $\kappa = 1$ and sufficiently large $k$, we obtain our nonlinear
eigenfunctions using a contraction map on the space $H_+^{s,\sw; 1,
  \smo}$. However, the nonlinear eigenfunction, or even the linear
eigenfunction $u_0$, does not lie in this space as its incoming
oscillations do not have the required decay. To deal with this, we
decompose $u_0$, the linear eigenfunction with incoming data $f$, into
two terms, $u_0 = u_+ + u_-$, where $u_-$ contains the leading incoming oscillation (which is the obstruction to membership in $H_+^{s,\sw; 1, \smo}$).   Consequently, the term
$u_+$ lies in $H_+ ^{s,\sw; 1, \smo}$ but $u_-$ does not. (Indeed, one
can think of $u_+$ as a sum of purely outgoing  terms plus the 
lower-order incoming terms, with additional decay, comprising $u_0$.)   We seek a nonlinear eigenfunction
satisfying
 $$
 u = u_0 + (\Delta - (\lambda + i0)^2)^{-1}  N[u],
 $$
 where $N(u)$ is the nonlinear term. Notice that, since the resolvent gains us two orders of smoothness, according to \eqref{eq:Fred-map}, $N$ can involve derivatives of $u$ up to order $2$.  Subtracting $u_-$ from both sides we have the equivalent equation 
  $$
 u - u_- = u_+ + (\Delta - (\lambda + i0)^2)^{-1}  N[u],
 $$
 and now defining $w = u - u_-$ we obtain 
 $$
 w = u_+ + (\Delta - (\lambda + i0)^2)^{-1}  N[u_- + w].
 $$
 Thus, it suffices to show that the map 
 \begin{equation}
 \Phi(w) := u_+ + (\Delta - (\lambda + i0)^2)^{-1}  N[u_- + w]
 \label{Phi}\end{equation}
 is a contraction on $H_+^{s,\sw; 1, \smo}$ when the norm of $w$ in this space is sufficiently small, which we show  provided the norm of $f$ in $H^{\smo + 2}$, $\smo > (n-1)/2$, is sufficiently small.

\subsection{Previous literature}
Standing wave solutions to nonlinear Schr\"odinger equations have been studied for a long time. The first studies were on finite-energy solutions, where the linearization at $u=0$ is the operator $\Delta + \lambda^2$ with $\lambda > 0$; this problem is of a different character, as the linearization at $u=0$ is an invertible operator. See \cite{BL1, BL2, Nehari} for classical work on this subject on Euclidean space, and \cite{CM, MT} for more recent works on hyperbolic and rotationally symmetric manifolds.  The more recent literature is vast and we make no attempt to review it. 

The first paper to study nonlinear Helmholtz eigenfunctions seems
to be \cite{Gut} by Guti\'errez, already discussed earlier in this
introduction. She was able to show that, for the cubic nonlinearity
and in dimensions 3 and 4, that there are nonlinear eigenfunctions
with arbitrary small incoming data $f \in L^2(\mathbb{S}^{n-1})$. We note
in passing that the restriction and uniform Sobolev estimates of \cite{GHS, GH} allow one to extend Guti\'errez'
method to all asymptotically conic manifolds.

There result of \cite{Gut} is a
perturbative result from the zero solution, as is ours
here. Non-perturbative  were found by Evequoz and Weth \cite{EW}, who
used mountain pass techniques to find solutions far from the zero
solution. These approaches have been extended in various ways in
\cite{MMP, LW}. In \cite{EP} the topology of the zero level sets of
bounded real solutions to $(\Delta - 1) u + u^3 = 0$ are studied. 

In the microlocal analysis literature, the underlying theory of real principal
type propagation in the setting of `scattering' pseudodifferential
operators, was developed by Melrose in \cite{RBMSpec}. The scattering calculus itself 
appeared earlier (at least on Euclidean space) in work of H\"ormander and Parenti, see for example \cite{Parenti}. A Fredholm theory for nonelliptic operators was developed by Vasy \cite{VD2013}  on anisotropic Sobolev spaces. This is elaborated and
explained in detail in his lecture notes \cite{va18}. His method
applies to operators that are of real principal type, except for
manifolds of radial points which have a particular structure. 
The first author with Haber and Vasy \cite{GHV} used this Fredholm
framework to study the Feynman propagator on asymptotically Minkowski
spaces and showed that the semilinear wave equation with polynomial
nonlinear is solvable for small data, using a setup very similar to
that considered here.  This latter result is an extension to a more
fundamentally microlocal setting of a previous result of Hintz and
Vasy \cite{HVsemi}.  Indeed, the latter two authors have developed a robust
microlocal analysis framework
which they use to study quasilinear wave
equations in various noncompact settings, see in particular
\cite{HVquasi, HV2018-kerr-de-sitter, HHV-Kerr-stable}.  In a recent series of papers \cite{Vasy-Resolvent-zero-energy, Vasy-LAP, Vasy-Res-near-zero-energy-1808.06123}, Vasy considers `second-microlocal' regularity for the Helmholtz operators, both at a fixed finite energy and near zero energy, which is very similar to our module regularity here. He proves  mapping
properties for the resolvent that overlap our result on the 
invertibility of the Helmholtz operator on spaces with module
regularity in Theorem \ref{thm:Fred prop} below.

\subsection{Outline of this paper}
In Section~\ref{sec:scat calc} we review the theory of pseudodifferential operators with variable order and define anisotropic Sobolev spaces. We discuss the geometry of the bicharacteristic flow of $\Delta - \lambda^2$ at spatial infinity and define the radial sets. We also discuss module regularity and define the corresponding spaces of functions. Finally, we consider algebra properties of these spaces with sufficient module regularity.

In Section~\ref{sec:Fredholm} we prove the invertibility of $\Delta -
\lambda^2$ acting between spaces as in \eqref{eq:Fred-map}. The proof
of this is at least implicitly contained in works of Vasy,
particularly his lecture notes \cite{va18}, but it is not explicitly
written out for this operator. Since, in addition, this is quite
recently developed technology and not standard, we have decided to
give at least some of the details, to make the paper more
self-contained.

In Section~\ref{sec:proof} we prove the main theorems, using the
technical preparation of the previous two sections.

\


\section{Scattering calculus}\label{sec:scat calc}

In this section, we discuss the technical tools that we need for the
proof of the main theorems. We begin by discussing the
pseudodifferential operators -- the scattering calculus -- used in the
proof, on $\RR^n$, and extend this in the following
subsection to asymptotically conic manifolds. We refer to \cite{va18}
and \cite{RBMSpec} for more detailed
treatment of the scattering calculus.


\subsection{The scattering calculus on $\RR^n$}
Throughout this paper, we denote Euclidean coordinates on $\RR^n$ by  $z = (z_1, \dots, z_n)$, and their dual coordinates by $\zeta = (\zeta_1, \dots, \zeta_n)$. We use the Japanese bracket $\ang{z}$ to denote $(1 + |z|^2)^{1/2}$. The Fourier transform, with H\"ormander's normalization, will be denoted $\mathcal{F}$, with inverse $\mathcal{F}^{-1}$:
\begin{equation}
\mathcal{F} f(\zeta) = \int e^{-i z \cdot \zeta} f(z) \, dz, \quad \mathcal{F}^{-1}  \tilde f(z) = (2\pi)^{-n} \int e^{i z \cdot \zeta} \tilde f(\zeta) \, d\zeta,
\end{equation}
We denote $-i \partial/\partial z_j$ by $D_{z_j}$, and use multi-index notation $D_z^\alpha$, $\alpha = (\alpha_1, \dots, \alpha_n) \in \mathbb{N}^n$ for higher-order derivatives, in the standard way. 

Pseudodifferential operators on $\RR^n$ are defined via their symbols, which are functions on $T^* \RR^n$. For sufficiently decaying symbols, say $a(z, \zeta) \in \mathcal{S}(T^* \RR^n)$, the corresponding pseudodifferential operator (defined by left quantization) is the operator with kernel 
\begin{equation}
\Op(a)(z,z') := (2\pi)^{-n} \int e^{i(z-z') \cdot \zeta} a(z, \zeta) \, d\zeta.
\end{equation}
This definition is extended to a larger class of symbols by integration by parts. 
The scattering calculus is obtained by letting $a$ lie in a (scattering) symbol class $S^{s,\sw}(T^* \RR^n)$. For fixed real numbers $s$ and $r$ this symbol class is defined by the estimates 
\begin{equation}
  \label{standard symbol est}
\forall \alpha, \beta \in \mathbb{N}^n, \ \exists C_{\alpha, \beta} < \infty \text{ such that }    \Big|  D_z^\alpha D_\zeta^\beta a(z,\zeta) \Big| \le
  C_{\alpha,\beta}  \ang{z}^{\sw-|\alpha|} \ang{\zeta}^{s-|\beta|}. 
\end{equation}
This is a rather restrictive class of symbols in which $z$ and $\zeta$
are treated symmetrically: differentiation in $\zeta$ leads to decay
in $\zeta$ and differentiation in $z$ leads to decay in $z$. It is in
the H\"ormander class of symbols \cite[Section 18.4]{Hor}
relative to the slowly varying metric
$$
\frac{dz^2}{\ang{z}^2} + \frac{d\zeta^2}{\ang{\zeta}^2}.
$$
The class of pseudodifferential operators of order $(s, r)$ is by definition the class of operators obtained 
from symbols $a \in S^{s,\sw}(T^* \RR^n)$ as above, and is denoted $\Psi_{\text{sc}}^{s,\sw}(\RR^n)$. These pseudodifferential operators form a bi-filtered algebra; concretely, the composition of an operator in $\Psi_{\text{sc}}^{s_1, \sw_1}(\RR^n)$ with an operator in $\Psi_{\text{sc}}^{s_2, \sw_2}(\RR^n)$ is an operator in $\Psi_{\text{sc}}^{s_1 + s_2, \sw_1 + \sw_2}(\RR^n)$. The symbol of the composition $\Op(a) \circ \Op(b)$ is given by 
$$
c(z, \zeta) = e^{i D_y \cdot D_\eta} a(z, \eta) b(y, \zeta) \Big|_{y = z, \eta = \zeta},
$$
and has an asymptotic expansion
\begin{equation}
c(z, \zeta) \sim  \sum_\alpha i^{|\alpha|} D_\zeta^\alpha a(z, \zeta) D_z^\alpha b(z, \zeta) / \alpha !
\label{symbol-exp}\end{equation}

Given this formula, and the decay of derivatives from \eqref{eq:3},  it is clear that the \emph{principal symbol}, which for $a \in S^{s,\sw}(T^* \RR^n)$ is its equivalence class in $S^{s,\sw}(T^* \RR^n) / S^{s-1,\sw-1}(T^* \RR^n)$, is multiplicative under composition. 
Notice that, unlike in the usual pseudodifferential calculus, here the
principal symbol is well defined up to symbols decaying (a full
integer order) faster in $z$, as well as decaying (a full integer
order) faster in
$\zeta$. That means that the principal symbol is, in effect,
completely well-defined at infinity for all finite frequencies
$\zeta$, and not just asymptotically as $|\zeta| \to \infty$, at least
in the case of classical symbols (discussed below).

We elaborate on this point. It is convenient in the scattering
calculus to view symbols on the compactification of $T^* \RR^n$. We
have already mentioned in the Introduction the radial compactification $\bbb$ of $\RR^n$.  This is obtained via the diffeomorphism
$\varphi \colon \mathbb{R}^n \lra \mathbb{B}^n$ from $\R^n$ to its unit ball, given by 
$$
z \mapsto \varphi(z)  = \frac{z}{1 + \ang{z}}  \in \mathbb{B}^n. 
$$
The closure of the image of this map is obviously the closed unit ball
$\overline{\mathbb{B}^n}$, and the map realizes $\mathbb{R}^n$ as the
interior of this compact manifold with boundary.  In keeping with
standard notation we write $\bbb \simeq \overline{\mathbb{B}^n}$ where
the notation indicates that we keep in mind the identification between
points in $\mathbb{R}^n$ with points in the ball
$\mathbb{B}^n$. 
We similarly radially compactify the fibre copy of $\RR^n$. Thus, we may understand the behavior of symbols by pulling them back via
$\varphi^{-1} \times \varphi^{-1} $ to $\bbb \times \bbb$. This is particularly helpful for classical symbols, which by definition take the form $\ang{z}^{\sw} \ang{\zeta}^s C^\infty(\bbb \times \bbb)$ (such functions automatically satisfy the symbol estimates \eqref{eq:3}). The class of such symbols is denoted $S_{\text{cl}}^{s,\sw}(T^* \RR^n)$. In particular, for classical symbols of order $(0, 0)$, the symbol is continuous up to the boundary of $\bbb \times \bbb$, and the principal symbol can be viewed as the boundary value of this symbol. Notice that this has two `components', one a function at fibre-infinity, that is, on $\bbb \times \mathbb{S}^{n-1}$, and one at `spatial infinity', that is, at $\mathbb{S}^{n-1} \times \bbb$. More 
generally, for an operator $A$ with classical symbol $a$ of order $(s, 0)$ (such as our Helmholtz operator $\Delta - \lambda^2$), the principal symbol is conveniently viewed as the combination of a fibre component, $\sigma_{\text{fiber},2,0}(A)(z,\zeta)$, which is homogeneous in $\zeta$ of degree $s$ (and is hence determined by $\zeta$ restricted to any sphere), and a base (or spatial) component, $\sigma_{\text{base},2,0}(A)(\omega,\zeta)$ where $\omega$ is the limiting value of $z/|z|$ on the sphere at infinity, and $\zeta \in \RR^n$. These have an obvious compatibility relation at the `corner' where $|z|$ and $|\zeta|$ are both infinite.  In particular, for the Helmholtz operator, the principal symbol is given by 
\begin{gather*}
\sigma_{\text{fiber},2,0}(\Delta - \lambda^2)(z,\zeta) \defeq \sum_{i,j} g_{ij} \zeta_i \zeta_j,  \\
\sigma_{\text{base},2,0}(\Delta - \lambda^2)(\omega,\zeta) \defeq \sum_{i,j} g_{ij} \zeta_i \zeta_j - \lambda^2.
\end{gather*}
It is important to understand that the base symbol need not be homogenous, and indeed is not homogeneous for the Helmholtz operator. 

Suppose $A \in \Psi_{\text{sc}}^{s,\sw}(\RR^n)$ has classical symbol. The elliptic
set of $A$, $\Ell_{s,\sw}(A) = \Ell(A)$, is the open subset of $\partial(\bbb \times \bbb)$ consisting of those points near which the principal symbol is at least as big as $c \ang{z}^\sw \ang{\zeta}^s$ for some $c > 0$. Its complement in $\partial(\bbb \times \bbb)$ is called the characteristic 
variety, $\Sigma_{s,\sw}(A) = \Sigma(A)$. The Helmholtz operator $\Delta - \lambda^2$ is elliptic at fibre-infinity, thus the characteristic variety is contained in the component at spatial infinity, and is given by
\begin{equation}
\Sigma(\Delta - \lambda^2) = \{ (\omega, \zeta) \in \mathbb{S}^{n-1} \times \RR^n \mid |\zeta| = \lambda \}.
\end{equation}
We also define the \emph{operator wavefront set} or \emph{microlocal support}, $\WF'(A)$ of $A$, to be the complement of the set of points $\q \in \partial(\bbb \times \bbb)$ such that, in a neighborhood $U$ of $\q$,  the full symbol $a(z,\zeta)$ satisfies \eqref{standard symbol est} for all $s, \sw \in \R$. Thus, intuitively speaking, $A$ is microlocally of order $-\infty$ in both the fibre and base senses away from $\WF'(A)$.

Returning to the composition formula \eqref{symbol-exp}, it is straightforward to see from this that the commutator of two pseudodifferential operators $A \in \Psi_{\text{sc}}^{s_1, \sw_1}(\RR^n)$ and $B \in \Psi_{\text{sc}}^{s_2, \sw_2}(\RR^n)$ is an operator $[A, B]$ in $\Psi_{\text{sc}}^{s_1 + s_2-1, \sw_1 + \sw_2-1}(\RR^n)$, with principal symbol given by the Poisson bracket of the symbols $a$ and $b$ of these operators:
\begin{equation}
\sigma_{\text{pr}}([A, B]) = \{ a, b \} \mod S^{s_1 +s_2 - 2, \sw_1 + \sw_2 - 2}(T^* \RR^n).
\end{equation}
We also recall that the Poisson bracket is given in terms of the Hamilton vector fields by 
\begin{equation}\label{pb}
\{ a, b \} = H_a (b) = - H_b(a), \quad H_a = \sum_j \Big( \frac{\partial a}{\partial \zeta_j} \frac{\partial}{\partial z_j}
- \frac{\partial a}{\partial z_j} \frac{\partial}{\partial \zeta_j} \Big).
\end{equation}
This is conceptually important for us in relation to the Fredholm estimates in Section~\ref{sec:Fredholm}. In the elliptic region, these Fredholm estimates are easy to obtain, but in a neighbourhood of the characteristic variety of $\Delta - \lambda^2$, they are obtained from positive commutator estimates, that is, from operators whose commutator with $\Delta - \lambda^2$ has positive principal symbol microlocally. Equation \eqref{pb} shows that this amounts to finding symbols $b$ such that $H_p(a)$ is positive, where $p = |\zeta|^2 - \lambda^2$ is the symbol of the Helmholtz operator. This then motivates considering the properties of the Hamilton vector field of $p$, and its flow lines (known as bicharacteristics), within the characteristic variety $\Sigma(\Delta - \lambda^2)$.

The Hamiltonian vector field $H_p$, for $p = |\zeta|^2 - \lambda^2$ the symbol of the Euclidean Helmholtz operator, is given by 
$$
\dot z = 2\zeta, \quad \dot \zeta = 0. 
$$
We would like to view this on the compactification $\bbb \times \bbb$, and investigate its behaviour in a neighbourhood of $\Sigma(\Delta - \lambda^2)$. To do this, we use polar coordinates, $(r, \omega)$, as before, and then choose arbitrary local coordinates $y$ on a patch of the sphere $\mathbb{S}^{n-1}$. We also write $x = r^{-1}$, which serves as a boundary defining function for spatial infinity. Let $(\nu, \eta)$ be the dual coordinates to $(r, y)$.  In these coordinates, the full symbol $p$ of $\Delta - \lambda^2$ takes the form 
$$
p(r,y, \nu, \eta) = \nu^2  -i (n-1)r^{-1} \nu + r^{-2}b_k \eta_k +  r^{-2} h^{jk} \eta_j \eta_k - \lambda^2,
$$
where $h^{jk} = h^{jk}(y)$ is the dual metric corresponding to the
standard round metric $h = h_{jk}(y)$, on $\mathbb{S}^{n-1}$, $b_k =
D_{y^j} h^{jk} + h^{jk}D_{y^j} \log( \sqrt{|\det h|})$, and we use the
summation convention.  
We now make the change of variables to $\mu_j = r^{-1} \eta_j$ as these quantities have uniformly bounded length as $r \to \infty$. In terms of $(\mu, \nu)$ the full symbol is 
$$
p(r,y, \nu, \mu) = \nu^2  -i (n-1)r^{-1} \nu + r^{-1}b_k \mu_k + h^{jk} \mu_j \mu_k - \lambda^2,
$$
and we see that the principal symbol at spatial infinity is 
\begin{equation}\label{flat-symbol}
\nu^2 + h^{jk} \mu_j \mu_k - \lambda^2 = \nu^2 + |\mu|^2_y - \lambda^2, 
\end{equation}
where $|\mu|^2_y :=h^{jk} \mu_j \mu_k$ is the metric function on $T^*\mathbb{S}^{n-1}$. Thus the characteristic set $\Sigma$ satisfies 
\begin{equation}\label{flat-Sigma}
\Sigma = \{x = 0, \, \nu^2 + |\mu|^2_y = \lambda^2\}.
\end{equation}

In the canonical coordinates $(r, y; \nu, \eta)$ we easily compute the Hamilton vector field of the principal symbol: 
\begin{equation}\begin{aligned}
\dot r = 2\nu, \quad &\dot y^l= 2 r^{-2} h^{lk} \eta_k, \\
\dot \nu = 2 r^{-3} h^{jk} \eta_j \eta_k, \quad &\dot \eta_l = -r^{-2} \frac{\partial h^{jk}}{\partial y^l} \eta_j \eta_k.
\end{aligned}\end{equation}
Changing to the variable $\mu$, and writing $x = r^{-1}$, the equations become 
\begin{equation}\begin{aligned}
\dot x = -2\nu x^2, \quad &\dot y^j = 2 x h^{jk} \mu_k, \\
\dot \nu = 2 x h^{jk} \mu_j \mu_k, \quad &\dot \mu_l = - 2x \nu \mu_l - x \frac{\partial h^{jk}}{\partial y^l} \mu_j \mu_k.
\end{aligned}\end{equation}
It is clear that this vector field vanishes to first order as $x \to 0$. Dividing by $x$ we obtain a rescaled Hamilton vector field that we denote by $\mathsf{H}_p$, taking the form 
\begin{equation}\begin{aligned}
\dot x = -2\nu x, \quad &\dot y^j = 2 h^{jk} \mu_k, \\
\dot \nu = 2 h^{jk} \mu_j \mu_k, \quad &\dot \mu_l = -  2\nu \mu_l - \frac{\partial h^{jk}}{\partial y^l} \mu_j \mu_k.
\end{aligned}\end{equation}
In the coordinates $(x, y, \nu, \mu)$ this is a smooth vector field on the compactification. We can write it using derivative notation as follows: 
\begin{equation}\label{flat-Hvf}
\mathsf{H}_p = -2\nu (x \partial_x + R_\mu) + 2|\mu|^2_y \partial_\nu + H_{\mathbb{S}^{n-1}} ,
\end{equation}
where $H_{\mathbb{S}^{n-1}}$ is the Hamilton vector of the round metric $|\mu|^2_y$ on $T^*\mathbb{S}^{n-1}$ and $R_\mu = \mu \cdot \partial_\mu$ is the radial vector field on the fibers of $T^*\mathbb{S}^{n-1}$. 
In these coordinates, and on $\Sigma$, we have $\mathsf{H}_p = 2\nu
R_\mu - 2|\mu|^2_y \partial_\nu + H_{\mathbb{S}^{n-1}}$.
 We can check
directly that $\mathsf{H}_p(\nu^2 + |\mu|^2_y) = 0$ and that the
precisely on the two `radial sets' 
\begin{equation}\label{def-radial}
\mathcal{R}_\pm \defeq \{|\mu|_y = 0 = x, \nu = \pm \lambda \}. 
\end{equation}

\begin{rem}
Notice that the incoming radial set $\mathcal{R}_-$ is a source, and the outgoing radial set $\mathcal{R}_+$ a sink, for the rescaled Hamilton vector field $\mathsf{H}_p$. Note, also, that the coefficient of $x \partial_x$ in $\mathsf{H}_p$ is $\pm \lambda$ at $\mathcal{R}_\pm$, hence always nonzero. This nonvanishing has the important consequence that we can find operators with positive commutators at $\mathcal{R}_\pm$, despite $\mathsf{H}_p$ vanishing there. In this sense the radial sets are `nondegenerate'. 
\end{rem}

\begin{figure}
\centering
\labellist
\small
\pinlabel $\mathcal{R}_-$  at  134 5
\pinlabel $\mathcal{R}_+$  at  133 182
\pinlabel \text{$\mathsf{l}_+ = -\frac{1}{2} - \delta$} at 13 184
\pinlabel \text{$\mathsf{l}_+ = -\frac{1}{2} + \delta$}  at 13 4
\pinlabel $\mu$ at 183  85
\pinlabel $\nu$ at 84 193
\endlabellist
\includegraphics[scale=.85]{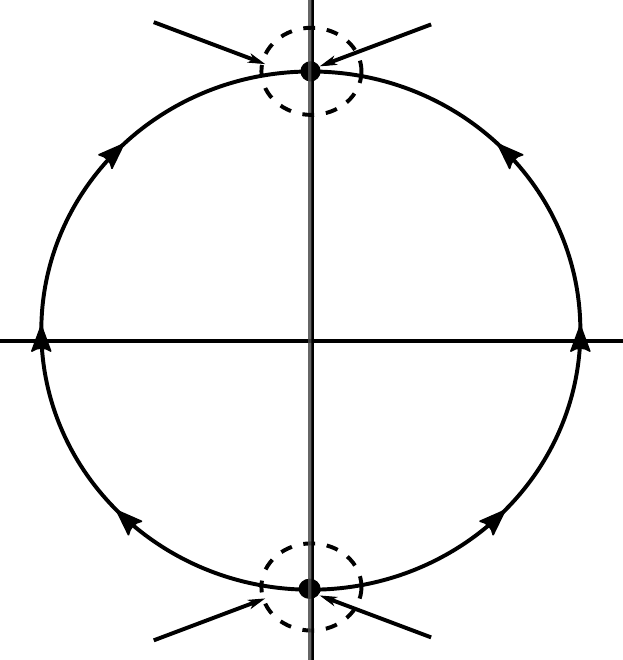}
 \caption{The bicharacteristic flow in the characteristic set \newline $\Sigma(P) = \{x = 0, \, \nu^2 + |\mu|^2_y = \lambda^2\}$ with $P=\Delta-\lambda^2$.}
 \label{fig: char set}
\end{figure}
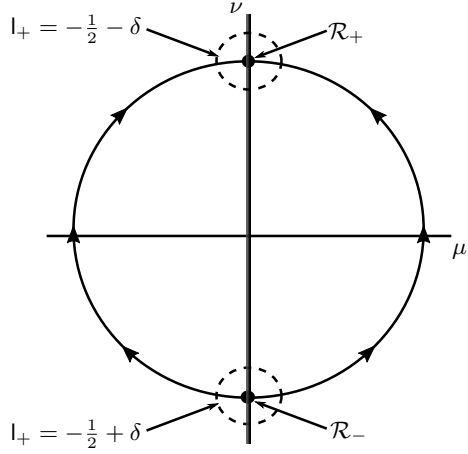


Up to this point, we have taken the spatial weight $\sw$ to be constant. To consider variable order spaces, we allow the spatial weight to itself be a classical symbol $\sfr$ of order $(0,0)$ (variable weights will always be written in boldface). Choosing an arbitrary small positive number $\boldsymbol\delta$, we define the symbol class $S^{s,\sfr}_{\boldsymbol\delta}(T^* \RR^n)$ by the estimates 
\begin{equation}
  \label{eq:3}
\forall \alpha, \beta \in \mathbb{N}^n, \ \exists C_{\alpha, \beta} < \infty \text{ such that }    \Big|  D_z^\alpha D_\zeta^\beta a(z,\zeta) \Big| \le
  C_{\alpha,\beta}  \ang{z}^{\sfr-(1 - \boldsymbol\delta)|\alpha| + \boldsymbol\delta |\beta|} \ang{\zeta}^{s-|\beta|} . 
\end{equation}
These are symbol estimates of type $(1 - \boldsymbol\delta, \boldsymbol\delta)$ in the $z$
variable (in the sense of H\"ormander), which are slightly `worse'
than the standard estimates of type $(1,0)$. The reason for including
this small loss is that the `classical' symbols corresponding to a
variable order take the form $\ang{z}^{\sfr} \ang{\zeta}^s$ times a
$C^\infty$ function on $\bbb \times \bbb$, and these symbols incur
logarithmic losses when differentating the $\sfr$ function. 

Changing to these symbol classes makes essentially no difference since we can take $\boldsymbol\delta$ arbitrary small, while pseudodifferential calculus, as is well known, works with inessential changes provided $\boldsymbol\delta < 1/2$. The only differences are that the principal symbol takes values in $S^{s,\sw}(T^* \RR^n) / S^{s-1 + \boldsymbol\delta,\sw-1+\boldsymbol\delta}(T^* \RR^n)$ instead of $S^{s,\sw}(T^* \RR^n) / S^{s-1,\sw-1}(T^* \RR^n)$, and the commutator $[A, B]$ above will have order $\Psi_{\text{sc}}^{s_1 + s_2-1 + \boldsymbol\delta, \sw_1 + \sw_2-1+\boldsymbol\delta}(\RR^n)$ instead of $\Psi_{\text{sc}}^{s_1 + s_2-1, \sw_1 + \sw_2-1}(\RR^n)$. 

\subsection{The scattering calculus on asymptotically conic manifolds}
We now work in the setting of asymptotically conic manifolds. Thus, let $M$ be a compact manifold with boundary, and $M^\circ$ its interior. Let $x$ be a boundary defining function for $M$ (meaning $\partial M = \{ x = 0\}$, $x$ vanishes simply at $\partial M$, and $x > 0$ on $M^\circ$), and $y$ coordinates on a patch $O$ of $\partial M$, extended to a collar neighbourhood $\{ x < c \}$ of $\partial M$, where $c >0$ is fixed and small. Given a scattering metric $g$ on $M^\circ$, we call $(x, y)$ as an `adapted coordinate system' near a boundary point $(0, y_0)$ with $y_0 \in O$ provided that $g$ takes the form \eqref{ac-metric} in this coordinate system on the patch $O$. This condition determines a metric $h$ on the boundary $\partial M$, such that $g$ is asymptotic to the conic metric $dr^2 + r^2 h$ as $r \to \infty$, where $r := 1/x$, as is clear from the equivalent expression \eqref{ac-metric2}. 

We now define scattering pseudodifferential operators of order $(s, \sw)$ on $M^\circ$. We do this by mimicking the behaviour of scattering symbols of order $(s, \sw)$ on $\RR^n$. To do this, we choose a diffeomorphism $\chi$ from a small open set $O \subset \partial M$ to an open set $O' \subset \mathbb{S}^{n-1}$, where $\mathbb{S}^{n-1}$ is viewed as the set of vectors of unit length in $\RR^n$. We then consider the diffeomorphism 
\begin{equation}\label{diffeo}
(x, y) \mapsto \frac{\chi(y)}{x} = r \chi(y) \in \RR^n.
\end{equation}
One can check that the norm of the derivative of this map is uniformly
bounded, where we measure with respect to the metric $g$ on $M^\circ$
and with respect to the Euclidean metric on $\RR^n$. We now define
suitable cotangent variables that are uniformly bounded (with respect
to the dual metric $g^*$). Let $(\nu, \eta)$ be the dual variables to
coordinates $(r, y)$, and define $\mu = r^{-1} \eta = x \eta$ as we
did in the previous section. Then $\nu$ is the symbol of $D_r = 
x^2 D_x$ and $\mu_j$ is the symbol of $ r^{-1} D_{y_j} = x
D_{y_j}$; since $(x^2 \partial_x, x \partial_{y_j})$ clearly form a
uniformly bounded, uniformly nondegenerate frame of functions with
respect to the metric $g$, these are uniformly bounded and uniformly
nondegenerate linear coordinates on the cotangent bundle, with respect
to the dual metric $g^*$. We define scattering symbols of order $(s,
\sw)$ on $T^* M^\circ$ to be functions $a$ satisfying usual symbolic
estimates of order $s$ away from $\partial M$, and near the
boundary satisfies
\begin{equation}
  \label{eq:3sc}
    \Big|  (x D_x)^j D_y^\alpha D_\nu^k D_\mu^\beta a(x,y,\nu, \mu) \Big| \le
  C_{j,k,\alpha,\beta}  x^{-\sw} \ang{(\nu, \mu)}^{s-k-|\beta|}
\end{equation}
for all $\alpha, \beta\in \mathbb{N}^{n-1}$. We will denote the class of such symbols by $S^{s, \sw}(\scTstar M)$. 
Scattering pseudodifferential operators $A \in \Psi_{\text{sc}}^{s,\sw}(M^\circ)$ are defined as follows: using the local diffeomorphism \eqref{diffeo}, the symbol is mapped (using the induced map on the cotangent bundle) to a symbol of order $(s,\sw)$ on $T^* \RR^n$; we then quantize to a pseudodifferential operator when $(z, z')$ are in the range of this diffeomorphism and pull back to $M^\circ$ by the same map \eqref{diffeo}. By covering $M^\circ$ with a finite number of coordinate charts and using a partition of unity, we get a globally defined operator. This quantization depends on the choice of charts, partition of unity, etc, but all choices lead to the same operator modulo an operator in $\Psi_{\text{sc}}^{s-1, \sw-1}(M^\circ)$ so this is of no importance. To complete the picture, we include in $\Psi_{\text{sc}}^{s,\sw}(M^\circ)$ all kernels $K(z,z')$ that are smooth and rapidly decreasing, with all derivatives, as the distance between $z$ and $z'$ tends to infinity in $M^\circ$. This definition is equivalent to the definition of the scattering pseudodifferential operators defined in \cite{RBMSpec} using the `scattering double space'. 


We see in \eqref{eq:3sc} that there are two types of vector fields (in terms of their behaviour near the boundary) that play a role in $M$. First, there are the \emph{b-vector fields}, which by definition are smooth vector fields that at the boundary are tangent to $M$. These are generated over $C^\infty(M)$ by $x \partial_x$ and $\partial_{y_j}$ near the boundary, and govern the regularity of scattering symbols in the spatial coordinates $(x, y)$ (this is called \emph{conormal regularity} in the microlocal literature). Second, there are the \emph{scattering vector fields}, which are just $x$ times b-vector fields, so generated by $x^2 \partial_x$ and $x \partial_{y_j}$. These have the property of generating, over $C^\infty(M)$, all smooth vector fields on $M$ that are uniformly bounded with respect to $g$. Scattering differential operators of order $(k, 0)$ are precisely differential operators of order $k$ that, near the boundary, can be written in terms of scattering vector fields with $C^\infty(M)$-coefficients. In the case that $(M^\circ, g)$, we can take the constant coefficient vector fields $\partial_{z_j}$ as generators of the scattering vector fields. Both will play a role in our analysis; the $s$ parameter in our pseudodifferential calculus is regularity with respect to scattering vector fields, while b-vector fields define module regularity (measured by the $\kappa$ and $k$ parameters, as discussed in the Introduction). 

Similarly to the Euclidean case, we can compactify the cotangent
bundle $T^* M^\circ$ in a way that mimics the compactification $\bbb
\times \bbb$ above. We have (by assumption) a compactification $M$ of
$M^\circ$. In the interior of $M^\circ$, we can compactify each
cotangent fibre radially. It only remains to say how the fibres are
compactified in the limit as we approach the boundary. We define a
\emph{scattering cotangent bundle}, denoted $\scTstar M$ over $M$,
which over the interior is naturally isomorphic to the usual cotangent
bundle, and has the property that, near $\partial M$, using adapted
coordinate system $(x, y)$, the corresponding coordinates $(\nu, \mu)$ (as defined above) are linear coordinates on the fibres of this bundle that remain valid uniformly up to the boundary $\partial M$. Compactifying each fibre radially gives us a compactification, denoted $\scTstarc M$, analogous to the square in Figure \ref{fig: cot bund}. This is a manifold with corners of codimension two. 
Clearly $x$ is a boundary defining function at spatial infinity. Let $\rho$ denote a boundary defining function for fibre-infinity --- we may take $\rho = \ang{(\nu, \mu)}^{-1}$ when $x$ is small. We will call $(x, y, \nu, \mu)$ adapted coordinates on the scattering cotangent bundle over the neighbourhood $\{ x < c, y \in O \}$ of $(0, y_0) \in \partial M$. 

\begin{figure}\label{fig:square}
\centering
\labellist
\pinlabel  \text{$\overline{{}^{\text{sc}}T^*M}$} at 10 178
\pinlabel  \text{${}^{\text{sc}}T_{\partial M}^*M = \{x = 0\}$} at -58 75
\pinlabel \text{${}^{\text{sc}}S^*M$} at 163 178
\pinlabel \text{$M \times \{0\}$} at 200 92
\pinlabel \text{$\mathcal{R}_-,\, \mathsf{l}_+ = -1/2 + \delta$} at  218 42
\pinlabel \text{$\mathcal{R}_+,\, \mathsf{l}_+ = -1/2 - \delta$} at  218 125
\endlabellist
\includegraphics[scale=.76]{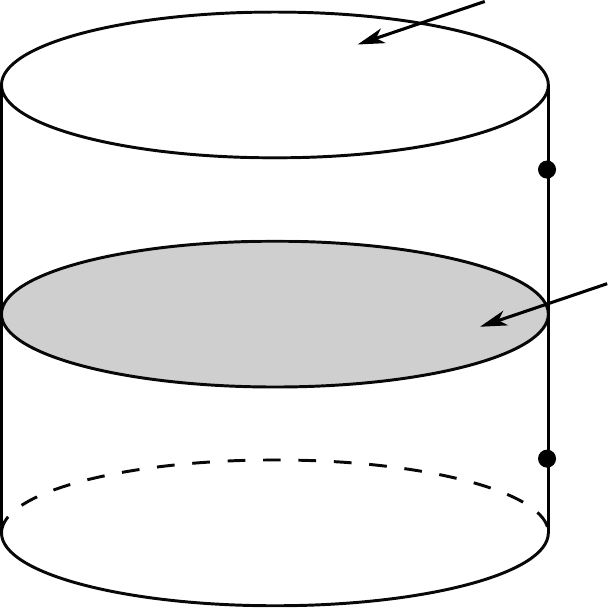}
 \caption{The radial sets in the compactified cotangent bundle. In the case $M^\circ = \RR^n$, ${}^{\text{sc}}T_{\partial M}^*M = \partial \overline{\RR^n} \times \overline{\RR^n} $ and ${}^{\text{sc}}S^*M =  \overline{\RR^n} \times \partial \overline{\RR^n}.$  }
 \label{fig: cot bund}
\end{figure}

We then can consider the subspace of operators $A$ with `classical' symbols $a$ of order $(s,\sw)$ that take the form $x^{-\sw} \rho^{-s}$ times a smooth function on $\scTstarc M$. The class of such symbols will be denoted $S^{s, \sw}_{\text{cl}}(\scTstar M)$.  For such operators, the principal symbol can be defined similarly to the classical case on $\RR^n$. Supposing for simplicity that $\sw=0$, we have a symbol at fibre-infinity, $\sigma_{\text{fiber},s,0}(A)$, which is a function on $\scTstar M$ homogeneous of degree $s$ on each fibre, and a symbol at spatial infinity, $\sigma_{\text{base},s,0}(A)$, which is the symbol $a$ restricted to $x=0$, a function on the scattering cotangent bundle restricted to $\partial M$ (which we denote $\scTstar_{\partial M} M$). These functions are well-defined, that is, they depend only on $A$, not on the particular quantization procedure (that is, the precise way of relating symbols $a$ and operators $A$). The two symbols have the compatibility condition that $\sigma_{\text{base},s,0}(A)$ is asymptotically homogeneous of degree $s$, that is, $\rho^{s}\sigma_{\text{base},s,0}(A)$ has a limit at $\rho = 0$,   and agrees at the corner, $x = \rho = 0$, with the limiting value of $\rho^s \sigma_{\text{fiber},s,0}(A)$. 

Now let $P$ denote the operator $\Delta_g + V - \lambda^2$ on $M^\circ$, where $\Delta_g$ is the (positive) Laplacian with respect to $g$ and $V \in x^2 C^\infty(M)$
is a real potential, vanishing to second order at $\partial M$. (The positive real number $\lambda$ will be fixed throughout.) Then the principal symbols of $P$, in adapted coordinates $(x, y, \nu, \mu)$  near the boundary, are
\begin{equation}
\sigma_{\text{fiber},2,0}(P) = \nu^2 + h^{jk}(y) \mu_j \mu_k,
\end{equation}
and 
\begin{equation}
\sigma_{\text{base},2,0}(P) = \nu^2 + h^{jk}(y) \mu_j \mu_k - \lambda^2. 
\end{equation}
This looks very similar to the form of the base symbol of the flat Laplacian on $\RR^n$ --- compare \eqref{flat-symbol}. It follows that the characteristic variety is given by \eqref{flat-Sigma}, just as in the flat case.
Moreover, exactly the same computation can be made as in the flat case to deduce that the Hamilton vector field of $p$, the symbol of $P$, takes the form 
\begin{equation}\label{ac-Hvf}
\mathsf{H}_p = -2\nu (x \partial_x + R_\mu) + 2|\mu|^2_y \partial_\nu + H_{\mathbb{S}^{n-1}} + xW ,
\end{equation}
where $W$ is a b-vector field (that is, tangent to $x=0$). That is, the Hamilton vector field takes the same form 
as \eqref{flat-Hvf}, up to a error  $x W$. In particular, the radial sets, where the Hamilton vector field vanishes, take the same form \eqref{def-radial} as in the Euclidean case.  \emph{This means that the microlocal analysis of the operator $P$ is no more complicated than that of the flat Laplacian on $\RR^n$} and means that, from this point of view, Theorem~\ref{thm:main2} is a very natural generalization of Theorem~\ref{thm:main1}. 

Our last topic to discuss is variable order scattering pseudodifferential operators. This is completely analogous to the case of variable order operators on $\RR^n$. We allow the spatial weight $r$ to be a classical symbol $\sfr$ of order $(0,0)$ on $\scTstarc M$, and allow a $\boldsymbol\delta$ loss in the symbol estimates. Thus, \eqref{eq:3sc} is replaced by 
\begin{equation}
  \label{eq:3sc var ord}
    \Big|  (x D_x)^j D_y^\alpha D_\nu^k D_\mu^\beta a(x,y,\nu, \mu) \Big| \le
  C_{j,k,\alpha,\beta}  x^{-\sfr - \boldsymbol\delta(j+k+|\alpha| + |\beta|)} \ang{(\nu, \mu)}^{s-k-|\beta|}
\end{equation}
and the rest of the theory proceeds as in the Euclidean case.

\subsection{Sobolev spaces of variable order} \label{Sobolev spaces of variable order}

We begin with the Euclidean case. The Sobolev spaces $H^{s,\sw}(\RR^n)$ are the usual weighted Sobolev spaces defined by 
$$
H^{s,\sw}(\RR^n) =  \{ f \in \mathcal{S}'(\RR^n) \mid\ang{D}^s  \ang{z}^\sw  f \in L^2(\RR^n) \}.
$$
Here $\ang{D}^s$ is the Fourier multiplier, given by $\mathcal{F}^{-1} \ang{\zeta}^s \mathcal{F}$. These can be equivalently defined by the condition that $f \in H^{s,\sw}(\RR^n)$ if and only if $Af \in L^2$ for all $A \in \Psi_{\text{sc}}^{s,\sw}(\RR^n)$; it is enough to require this for just one $A$ that is `totally elliptic', that is, both its symbol at fibre-infinity and at spatial infinity are everywhere elliptic, or equivalently, $\Ell(A) = \partial(\bbb \times \bbb)$. We use this characterization to define the Sobolev spaces for variable orders: we say that $f \in H^{s, \sfr}(\RR^n)$ if $Af \in L^2$ for all $A \in \Psi_{\text{sc}}^{s, \sfr}(\RR^n)$; again, it is enough to require this for one totally elliptic operator. If in addition, $A$ is invertible, that is, $A^{-1} \in \Psi_{\text{sc}}^{-s, -\sfr}(\RR^n)$, then the norm of $f$ in $H^{s,\sfr}(\RR^n)$ can be taken to be $\| Af \|_{L^2}$.  

The Sobolev spaces for an asymptotically conic manifold are defined
analogously: we say $f \in H^{s, \sfr}(M^\circ)$ if $Af \in L^2$
for all $A \in \Psi_{\text{sc}}^{s, \sfr}(M^\circ)$;  it is enough to require this
for one totally elliptic operator. If in addition, $A$ is invertible,
that is, $A^{-1} \in \Psi_{\text{sc}}^{-s, -\sfr}(M^\circ)$, then the norm of $f$
in $H^{s,\sfr}(M^\circ)$ can be taken to be $\| Af \|_{L^2}$.  

Pseudodifferential operators of variable order act on Sobolev spaces with variable order in the expected way: if $A \in \Psi_{\text{sc}}^{s, \sfr}(M^\circ)$, then $A$ is a bounded linear map from $H^{s', \sfr'}(M^\circ)$ to $H^{s'-s, \sfr' - \sfr}(M^\circ)$. Also, the dual space of $H^{s, \sfr}(M^\circ)$ is $H^{-s, -\sfr}(M^\circ)$. The duality between these two spaces can be realized by choosing any invertible $A \in \Psi_{\text{sc}}^{s, \sfr}(M^\circ)$. Then for $u \in H^{s, \sfr}(M^\circ)$, $v \in H^{-s, -\sfr}(M^\circ)$ we define 
$$
(u, v) \defeq \ang{ A u, A^{-1} v }_{L^2},
$$
It is easy to check that this pairing is independent of the particular invertible operator $A \in \Psi_{\text{sc}}^{s, \sfr}(M^\circ)$. 

We now define the spaces $ \mathcal{X}^{s,\sfr_+}$ and $\mathcal{Y}^{s,\sfr_+}$. We have already discussed in the Introduction that $\mathcal{Y}^{s,\sfr_+}$ is precisely the variable order Sobolev space $H^{s, \sfr_+}$ as defined above, for a variable spatial weight $\sfr_+$ with specific properties on $\Sigma(P)$ (the behaviour away from $\Sigma(P)$ is unimportant). First, we require that, for some small $\delta > 0$, 
\begin{equation}\label{weight-radial}
\begin{gathered}
 \text{$\sfr_+$ takes values in $[-1/2 - \delta, -1/2 + \delta]$, } \\
 \text{and is equal to $-1/2 \mp \delta$ in a neighbourhood of $\mathcal{R}_\pm$.}
 \end{gathered}
\end{equation}
This ensures elements of $\mathcal{Y}_+^{s,\sfr_+}$ are permitted to have outgoing oscillations of the form $r^{-(n-1)/2} e^{i\lambda r}$ but not incoming oscillations of the form $r^{-(n-1)/2} e^{-i\lambda r}$. Second, we require that 
\begin{equation}\label{weight-monotone}
\text{ $\sfr_+$ is nonincreasing along the Hamilton flow of $P$ within $\Sigma(P)$.}
\end{equation}
Since bicharacteristics within $\Sigma(P)$ start at $\mathcal{R}_-$ and end at $\mathcal{R}_+$, these two conditions are compatible. We also define
\begin{equation}\label{weight-pm}
\sfr_- = -1 - \sfr_+.
\end{equation}
This automatically means that $\sfr_-$ has analogous properties to $\sfr_+$ with the incoming and outgoing radial sets swapped. In particular, we have 
\begin{equation}\label{weight-radial-minus}\begin{gathered}
 \text{$\sfr_-$ takes values in $[-1/2 - \delta, -1/2 + \delta]$,} \\
\text{  is equal to $-1/2 \pm \delta$ in a neighbourhood of $\mathcal{R}_\pm$,} \\
 \text{ and is nondecreasing along the Hamilton flow of $P$ within $\Sigma(P)$.}
 \end{gathered}\end{equation}

\begin{rem}
Condition \eqref{weight-monotone} is imposed so that regularity of approximate solutions of $P u = 0$ can be propagated from the incoming radial set $\mathcal{R}_-$ towards the outgoing radial set $\mathcal{R}_+(\lambda)$, in the following section. 
\end{rem}

We then define the spaces $\mathcal{X}^{s,\sfr_\pm}$ by  \begin{equation}
  \label{eq:44}
  \mathcal{X}^{s,\sfr_\pm} \defeq \large\{ u \in H^{s,\sfr_\pm} \mid  
  Pu \in H^{s - 2, \sfr_\pm + 1} \large\},
\end{equation}
with norm 
\begin{equation}
\| u \|^2_{\mathcal{X}^{s,\sfr_\pm}} = \| u \|^2_{H^{s,\sfr_\pm}} + \| Pu \|^2_{H^{s - 2, \sfr_\pm + 1}}.
\end{equation}

\subsection{Test modules and Sobolev spaces with module regularity}
We next introduce the `test modules' with respect to which we will assume further
differentiability.  A test module $\mathcal{M}$, as defined in \cite{HMV2004}, is a subspace of 
$\Psi_{\text{sc}}^{1,1}(\RR^n)$, or $\Psi_{\text{sc}}^{1,1}(M^\circ)$ in the general case, that is closed under commutators, that contains the identity and is a module over $\Psi_{\text{sc}}^{0,0}$. (Here we adapt the definition of \cite{HMV2004} slightly to allow order 1 in the fibre as well as the spatial slot, as is convenient here.) We shall also work only with finitely generated modules $\mathcal{M}$, which have the form 
$$
\mathcal{M} = \Big\{ \sum_{j=0}^N C_j A_j \mid C_j \in \Psi_{\text{sc}}^{0,0}(M^\circ) \Big\}, 
$$
for some fixed finite set $A_0 = \Id, A_1, \dots A_N \subset \Psi_{\text{sc}}^{1,1}(M^\circ)$, the \emph{generators} of the module, which should be closed under taking commutators in the sense that 
$$
[A_j, A_k] = \sum_{l=0}^N E_{jkl} A_l \quad E_{jkl} \in \Psi_{\text{sc}}^{0,0}(M^\circ).
$$


The most important modules for us will be the modules
$\mathcal{M}_{\pm \lambda}$ defined by the
characteristic condition
\begin{equation}
\mathcal{M}_\pm := \mathcal{M}_{\pm \lambda} = \{ A \in \Psi_{\text{sc}}^{1,1} \colon \mathcal{R}_\pm
\subset \Sigma_{1,1}(A) \},\label{eq:real-modules}
\end{equation}
where $\mathcal{R}_\pm$ are the radial sets in \eqref{def-radial}.
Before discussing these in detail, we note that they are examples of
what will become (for us) a useful general class of modules
$\mathcal{M}_\mmuu$ also defined by a characteristic condition given
in terms of a real parameter $\mmuu$.  Recalling the coordinates $(x,
y, \nu, \mu)$ defined near the boundary $\p M = \{ x = 0 \}$ with
respect to local coordinates $y$ on an open set $O \subset \p M$, with
$(x,y)$ adapted coordinates on $M$, let
$$
\mathcal{R}_\mmuu = \{ x = 0 , \ |\mu|_y = 0,\ \nu = \mmuu \}.
$$
Then we define
\begin{equation}\begin{gathered}
\mathcal{M}_\mmuu := \{ A \in \Psi_{\text{sc}}^{1,1} : \mathcal{R}_\mmuu \subset
\Sigma_{1,1}(A) \}.
\end{gathered}\end{equation}
Writing $\Diff^{1,1} \subset \Psi_{\text{sc}}^{1,1}$ for the subspace of \emph{differential} operators, one can then choose a generating set for these modules containing three
types of operators
\begin{enumerate}
\item   $A_N \in \Diff^{1,1}(M)$ and,
  $$
  A_N = r(D_r - \mmuu) = - x D_x - \frac{\mmuu}{x}, \mbox{ on } x \le c,
  $$
  where $x \le c$ on our fixed collar neighborhood of $\p M$.
\item $A_j \in \Diff^{1,1}(M)$, $j = 1, \dots, N_1$ which are
  \emph{purely angular} in the sense that they are in the
  $C^\infty(M)$ linear span of the $D_{y_p}$ for some coordinates $(y_p)$ on $\p M$ and adapted coordinates $(x,y)$ on $M$, and lastly
\item $A'_k \in \Diff^1(M^\circ)$, $k = 1, \dots, N_2$ which
  are supported in $\{ x > c \}$, so
  in particular $A'_k \in \Psi_{\text{sc}}^{1, \infty}$,
\end{enumerate}
and using these one has
\begin{equation}\begin{gathered}\label{eq:og-mod-def}
\mathcal{M}_\mmuu \defeq \Big\{ C_0 + \sum_{j=1}^{N_1} C_j A_j +
\sum_{k=1}^{N_2} C_k' A_k' + C_N A_N\mid C_j, C_k' \in \Psi_{\text{sc}}^{0,0}(M^\circ) \Big\}. 
\end{gathered}\end{equation}
Here $N_1 + N_2 + 1= N$, i.e.\ there are $N + 1$ total generators
including $A_0$.
(To be overly concrete, one can cover the boundary with $m$
coordinate charts $O_m$ with coordinates $y_i^{(q)}, i = 1, \dots, n-1,
q = 1, \dots m$, with the $(y_i^{(q)})_i$ coordinates on $O_q$, and
let $A_j = A_{i, q} = \chi_q \p_{y_i^{(q)}}$  for the $\chi_q$ a
partition of unity subordinate to $O_1, \dots, O_m$, and then choose the $A'_k$
be any finite family of vector fields for which, for all $q \in T^*
M^{\circ}$ with $x(q) > c$, there is a $A'_k$ with $\sigma(A_k')(q)
\neq 0$.)
In the case $\mmuu = 0$, these generators form a basis (over
$C^\infty(M)$) of the b-vector fields, that is, all vector fields
tangent to the boundary of $M$ (in the case of $\RR^n$, this includes
all constant coefficient vector fields times a factor $r$).

We note that in particular the $A_j$ and $A_k'$ are all elements of
$\mathcal{M}_\mmuu$ (this is not a requirement for test modules in
other contexts, e.g.\ \cite{Haber-Vasy:Radial}).  The operators $A_N$
in (i)
and $A_j$ in (ii) taken together have the feature that for any point
$\mathsf{q} \in \Sigma_{2,0}(P) \setminus \mathcal{R}_+$ there is an element in
the module $A = \sum C A_N + \sum_{j = 1}^{N_1} C_j A_j$ such that
$\sigma_{1,1}(A)(\mathsf{q}) \neq 0$, i.e.\ $\mathsf{q} \in \Ell_{1,1}(A)$.  Indeed,
using adapted coordinates $(x, y, \nu, \mu)$ for one of our coordinate
charts $O_q$ and writing $\mathsf{q} = (0, y,
\nu, \mu)$, $\nu^2 + |\mu|^2 = \lambda$, if $\mu \neq 0$ then we can
choose a vector field $V := \chi c_i D_{y_i}$ with $\sigma_{\p
  M}(V)(\mathsf{q}) \neq 0$ where this is the standard symbol of a
vector field on a closed manifold and $\chi$ is supported in $O_\q$ and
$c_i \in \mathbb{R}$.  Then $\sigma_{1,1}(V) =
\sigma_{1,0}(r^{-1}V) = (|\nu|^2 + |\mu|^2)^{-1} \chi \sum c_i \mu_i$, and thus
$$
\sigma_{\text{base}, 1,1}(V) = \chi c_i \mu_i,
$$
in particular $\sigma_{\text{base}, 1, 1}(V)(\mathsf{q}) \neq 0$.
Similarly
$$
\sigma_{\text{base}, 1, 1}(A_N) = \nu - \lambda,
$$
so $A_N$ is elliptic on the whole of $\mathcal{R}_-$.

In the case that $\mmuu = \pm \lambda$, the test module
$\mathcal{M}_{\pm \lambda}$ will be abbreviated $\mathcal{M}_\pm$, as
in \eqref{eq:real-modules}. These test modules can be characterized as all scattering pseudodifferential operators of order $(1,1)$ that are characteristic at the radial set $\mathcal{R}_\pm(\lambda)$ given by \eqref{def-radial}. 
Analytically, the significance of these two modules is that all the generators annihilate the corresponding radial oscillation, $e^{\pm i \lambda r}$.

For any $\mmuu \in \RR$, we can then define the weighted Sobolev
spaces $H^{s,\sw;\kappa}_{\mathcal{M}_\mmuu}$ with module regularity of order
$\kappa$ with respect to $\mathcal{M}_\mmuu$. 
These consist of functions in
$H^{s,\sw}$ that remain in this space under the composition of any $\kappa$
elements in $\mathcal{M}_\mmuu$. For $\mmuu = \pm \lambda$, these
spaces will be denoted $H_{\pm }^{s,\sw; \kappa}$ for brevity.  A distribution $u$ will lie
in $H_{\mathcal{M}_\mmuu}^{s,\sw;\kappa}$ if and only if $u \rvert_{x > c}
\in H^{s + \kappa}$ and for any adapted coordinate system $(x, y)$ we have
$$
(r(D_r - \mmuu))^j D^\beta_y  u \in H^{s,\sw} \text{ whenever } j + |\beta|
   \le \kappa.
   $$
To impart the structure of a Hilbert space to $H^{s,\sw;
  \kappa}_{\mathcal{M}_\mmuu}$ we use the generators $\mathsf{A}_j, j = 0,
\dots, N$ of $\mathcal{M}_\mmuu$, where the $\mathsf{A}_j$ run over
all the $A_j$ and $A_k'$ in the definition of $\mathcal{M}_\mmuu$
above.  Then, using standard
multi-index notation $\mathsf{A}^\alpha = \mathsf{A}_0^{\alpha_0} \dots \mathsf{A}_N^{\alpha_N}$,
$\alpha \in \Nat^{N + 1}$, we define\begin{equation}  \label{eq:mod reg sob spaces}
  \begin{gathered}
   H_{\mathcal{M}_\mmuu}^{s,\sw;\kappa}  \defeq \big\{u \in H^{s,\sw} : \mathsf{A}^\alpha u
   \in H^{s,\sw} \text{ whenever } |\alpha| \le \kappa \big\}, \\
    \|u\|^2_{H^{s, \sw; \kappa}_{\mathcal{M}_\gamma}} \defeq  \sum_{|\alpha| \le \kappa}
    \|\mathsf{A}^\alpha u  \|^2_{H^{s, r}}.
  \end{gathered}
\end{equation} 

In particular, when $\mmuu=\pm\lambda$, we define
\begin{equation}  \label{eq:mod reg sob spaces1}
  \begin{gathered}
H_{\pm}^{s,\sw;\kappa}  \defeq H_{\mathcal{M}_{\pm\lambda}}^{s, \sw ;\kappa}  .
  \end{gathered}
\end{equation}

As well as the modules $\mathcal{M}_\mmuu$, we shall need to consider
the smaller module $\mathcal{N} \subset \mathcal{M}_\mmuu$ (for any
$\mmuu$) generated only by the purely angular and purely interior
derivatives, i.e. in the notation preceding  \eqref{eq:og-mod-def},
only the generators $A_0 = \Id$, the purely angular $A_j$ for $j = 1,
\dots, N_1$ and the interior $A_k'$, $k = 1, \dots,
N_2$, 
\begin{equation}
  \label{eq:y-module}
  \mathcal{N} = \Big\{ C_0 +  \sum_{j=1}^{N_1} C_j A_j +
\sum_{k=1}^{N_2} C_k' A_k' \mid C_j, C_k' \in \Psi_{\text{sc}}^{0,0}(M^\circ) \Big\}. 
\end{equation}
In direct analogy with $H^{s,\sw;\kappa}_{\mathcal{M}_\mmuu}$, writing the
generators of $\mathcal{M}_{\mmuu}$ as $\mathsf{A}_j$, $j = 0, \dots, N$ and
those of $\mathcal{N}$ as $\mathsf{B}_k$, $k = 0, \dots, N-1$, we define $u \in H_{\mathcal{M}_\mmuu}^{s,\sw; \kappa, \smo}$ if an only if $u \rvert_{x > c}
\in H^{s + \kappa + \smo}$, and for any adapted coordinate system $(x, y)$ we have

\begin{equation}  \label{eq:double-module-sob}
 \begin{gathered}
  H^{s, \sw; \kappa, \smo}_{\mathcal{M}_\mmuu} \defeq \{ u \in H^{s,\sw} : \mathsf{A}^\alpha \mathsf{B}^\beta u
  \in H^{s,\sw} , \  |\alpha| \le \kappa,  |\beta| \le k \}, \\
      \|u\|^2_{H^{s, \sw; \kappa, \smo}_{\mathcal{M}_\mmuu} }\defeq  \sum_{ |\alpha| \le \kappa,  |\beta| \le \smo}
    \|\mathsf{A}^\alpha \mathsf{B}^\beta u  \|^2_{H^{s, \sw}}.
    \end{gathered}
\end{equation}

In particular, for $\mmuu=\pm\lambda$, we put
\begin{equation}  \label{eq:double-module-sob1} 
 \begin{gathered}
H_{\pm}^{s,\sw;\kappa,\smo}  \defeq H_{\mathcal{M}_{\pm\lambda}}^{s,\sw;\kappa,\smo}  .
  \end{gathered}
\end{equation}

Notice that we have the simple relation between these spaces:
\begin{lemma}\label{lem:mu-change}
  Let $\mmuu, \mmuu' \in \mathbb{R}$.
  \begin{equation}
    \label{eq:changing-modules}
H^{s,\sw;\kappa, \smo}_{\mathcal{M}_{\mmuu'}} =     e^{i (\mmuu' - \mmuu)r }H_{\mathcal{M}_{\mmuu}}^{s,\sw;\kappa,\smo}.
  \end{equation}
  \end{lemma}
  \begin{proof}
    This follows directly from the relation  $(D_r - \mmuu') e^{ i (\mmuu' - \mmuu)r} u = e^{ i( \mmuu' - \mmuu)r}(D_r - \mmuu) u$.
  \end{proof}
  
  We also note without proof the simple mapping property of scattering pseudodifferential operators on these spaces:
  
  \begin{lemma}\label{lem:psdo-mapping}
  Let $A \in \Psi_{\text{sc}}^{m, \ell'}(M)$. Then $A$ is a bounded operator
  \begin{equation}
  A : H_\pm^{s, \sw; \kappa, k} \to H_\pm^{s-m, \sw - \ell'; \kappa, k}.
  \end{equation}
  \end{lemma}

The module  $\mathcal{N}$ enjoys an important
vanishing property at the radial sets, which we describe now. Returning to the general situation for a moment, let
$\mathcal{M}$ be any test module, generated by $A_0 = \Id, A_1, \dots, A_N$, and let $P$ be the Helmholtz operator on $\RR^n$ or the generalized Helmholtz operator on $M^\circ$. 
 We shall say that $\mathcal{M}$ is
 \textbf{$P$-critical} at the subset $S \subset \Sigma(P)$ if there exist $C_{jk} \in
   \Psi_{\text{sc}}^{1,0}(M)$ with
  \begin{equation} \label{off diag symbols vanish}
i x^{-1} [A_j, P] = \sum_{k = 0}^N C_{jk} A_k \quad \text{ with } \quad 
 \sigma_{1,0}(C_{jk}) = 0 \text{ on $S$ for all $j, k$. }
\end{equation}
%


\begin{lemma}\label{thm:posititivy-criticality}
  The module
  $\mathcal{N}$ is $P$-critical at both radial sets $\mathcal{R}_+
  \cup \mathcal{R}_-$.
\end{lemma}

\begin{proof}
Consider the commutators with
generators of $\mathcal{N}$. The condition is trivially satisfied for generators microsupported away from the boundary of $M$, since then \eqref{off diag symbols vanish} is vacuous for $S = \mathcal{R}_\pm$, and the commutator is a differential operator of order $2$, so can certainly be written in the form \eqref{off diag symbols vanish}. 

It remains to check the commutator with tangential derivatives $D_{y_j}$ near the boundary. 
The operator $P$ takes the form near the boundary, in adapted coordinates $(x, y)$, 
\begin{equation}\label{P-adapted}
(x^2 D_x)^2 + i(n-1) x (x^2 D_x) + x^2 Q + x^2 \tilde Q-\lambda^2,
\end{equation}
where $Q$ is a differential operator of order $2$ in the tangential
derivatives $\partial_{y_j}$ with coeff\-icients smooth on
$M$ (in fact, it is precisely the tangential Laplacian, $\Delta_{h(x)}$ for $\partial M$ with metric $h(x)$), and $\tilde Q  \in \Diff^{1,0}$ (in fact, it has a purely radial component, obtained by differentiating $h(x)$ in $x$). 
The commutator with $D_{y_j}$ is an operator of the form 
$$
ix^{-1} [D_{y_j}, P] = i x [D_{y_j}, Q] + ix[D_{y_j}, \tilde Q].
$$
The term $ix [D_{y_j}, Q]$ is a second order differential operator in the tangential variables times $x$, therefore takes the form of a sum of terms of
the form $C_{jk} A_k$, for $1 \leq k \leq n-1$, where $C_{jk}$ is
equal to $\sum_l b_l x \partial_{y_l}$, where $b_l \in
C^\infty(M)$. Since the symbol of $x \partial_{y_l}$ vanishes at
$\mathcal{R}_+$, this satisfies the conditions of $P$-criticality. The term $ix[D_{y_j}, \tilde Q]$ is a scattering differential operator of order $(1,0)$ times $x$, which also satisfies the condition for $P$-criticality.

\end{proof}

By analogy with the spaces $\mathcal{X}^{s,\mathsf{l}_+}$ and $\mathcal{Y}^{s,\mathsf{l}_+}$, we define
\begin{equation}
\mathcal{Y}_+^{s, \sw ; \kappa, \smo} = H_+^{s, \sw ; \kappa, \smo} 
\end{equation}
and
\begin{equation}
  \label{eq:Xsrkl}
  \mathcal{X}_+^{s,\sw; \kappa, \smo} \defeq \large\{ u \in H_+^{s,\sw; \kappa, \smo} \mid  
  Pu \in H_+^{s - 2, \sw + 1; \kappa, \smo} \large\}
\end{equation}
with norm 
\begin{equation}
\| u \|^2_{\mathcal{X}_+^{s, \sw; \kappa, \smo}} = \| u \|^2_{H_+^{s, \sw; \kappa, \smo}} + \| Pu \|^2_{H_+^{s - 2, \sw + 1; \kappa, \smo}}.
\end{equation}

The main technical result of this paper is the following mapping property of the Helmholtz operator $P$ on these spaces with module regularity: 

\begin{thm}\label{thm:Fred prop}
  Let $s, \sw \in \mathbb{R}$, and assume $\sw
  \in (-3/2, -1/2)$.  For any natural numbers $\kappa \ge 1$, $k \ge 0$, the map
  \begin{equation}
    \label{eq:real-mod-map}
      P \colon \mathcal{X}^{s,\sw; \kappa, \smo}_{\pm} \lra \mathcal{Y}^{s-2,
    \sw+1; \kappa, \smo}_{\pm}
  \end{equation}
is an isomorphism of Hilbert spaces.  In particular the inverse map,
i.e., the outgoing ($+$), resp. incoming ($-$), resolvent is bounded
as a map
\begin{equation}
  \label{eq:resolvent-bound}
  R(\lambda \pm i0) \colon H^{2-s, \sw + 1; \kappa, \smo}_\pm \lra H^{s, \sw; \kappa, \smo}_\pm.
\end{equation}
  \end{thm}

We defer the proof to Section~\ref{sec:Fredholm}.

  \subsection{Multiplicative properties of weighted Sobolev spaces with module
  regularity}
  
  We prove multiplicative properties of weighted Sobolev spaces on $\RR^n$, or more generally on asymptotically conic manifolds, with additional module regularity. We use the module $\mathcal{M}_0$ generated by b-vector fields, as that gives us the best multiplicative properties, and deduce more general multiplicative properties as a corollary.

\begin{lemma}\label{thm:algebra}
Let $\sw, \sw' \in \RR$, $s, \kappa, k \in \mathbb{N}_0$. If $\kappa \ge 1$ and $k \ge (n-1)/2$,
  multiplication on $C^\infty_c(M)$ extends to a bounded bilinear map 
    \begin{equation}
    \label{eq:mult}
    H^{s,\sw;\kappa, \smo}_{\mathcal{M}_0} \cdot     H^{s,\sw';\kappa, \smo}_{\mathcal{M}_0} \lra
    H^{s,\sw + \sw' + n/2;\kappa, \smo}_{\mathcal{M}_0}.
  \end{equation}
\end{lemma}
Before the proof, we make some remarks concerning  spaces of distributions on $\mathbb{R}^n$ whose
regularity in an $L^2$-based Sobolev sense is a given order, $\kappa$,
with some additional order $k$ of regularity only in certain directions. 
Write $\RR^n = \RR^{d} \times  \mathbb{R}^{n - d}$, $z = (z',
z'')$, where $z' \in \RR^d$ and $z'' \in \RR^{n - d}$.  Define
\begin{equation}\label{eq:basicmodulereg}
\cY^{\kappa, k}_d(\RR^{d} \times \RR^ {n - d}) = \{ u :  \la \zeta \ra^\kappa \la \zeta^{''}
  \ra^k \hat{u}  \in L^2 \},
\end{equation}
with $\zeta = (\zeta', \zeta'')$.  Thus distributions in $\cY^{\kappa, k}$ have $\kappa$
total derivatives in $L^2$
and an additional $k$ derivatives in $z''$ in $L^2$.  For our
purposes, below we extend this definition to the case where one factor
is a closed manifold $\boldsymbol N$ of dimension $n-d$; for $\kappa, k \in \mathbb{N}_0$, and $d\mu_N$ a measure on $\boldsymbol N$,  
\begin{multline*}
\cY^{\kappa, k}_d(\RR^{d}_z \times \boldsymbol N) = \{ u \in L^2(\RR^{d}_z \times \boldsymbol N, dz \, d\mu_N) \mid D_z^\alpha A u \in
L^2(\RR^{d}_z \times \boldsymbol N, dz \, d\mu_N) , \\ \mbox{ for all } |\alpha|
\le \kappa, \ A \in \Diff^k(\boldsymbol N) \},
\end{multline*}
and below we will take $d=1$ and $\boldsymbol N = \p M$.

We will use
the following lemma, which is proven in \cite[Lemma~4.4]{HVsemi}
  \begin{lemma}\label{thm:basicalgebra}
    Let $\kappa, k \in \mathbb{R}$.  If $\kappa > d /2$ and $k \ge (n -
    d)/2$ then $\cY_d^{\kappa, k}(\RR^{d}_z \times \boldsymbol N)$ is an
    algebra. 
  \end{lemma}
%
  

\begin{proof}[Proof of Lemma \ref{thm:algebra}] 
First, we suppose $\sw = \sw' = s = 0$. Let $u, v \in C^\infty_c(M^\circ)$, and let $\chi = \chi(x) \in C^\infty(M)$ be a cutoff function, identically one near the boundary, and supported in the collar neighbourhood $\{x < c\}.$ We decompose the product $uv$ as 
\begin{equation*}
uv  = (1 - \chi^2) uv + \chi^2 uv,
\end{equation*} 
 and bound the $H_{\mathcal{M}_0}^{0,n/2;\kappa,\smo}$ norm of both pieces. 
 
 Since $\kappa + k > n/2$ and $1 - \chi^2$ has compact support, the standard algebra property for $H^{\kappa + \smo}(M^\circ)$, along with the equivalence of the $H^{\kappa + \smo}$ and $H_{\mathcal{M}_0}^{0,\tilde{\sw};\kappa,\smo}$ norms (any $\tilde{\sw} \in \R$) on a fixed compact subset of $M^\circ$, yields
 \begin{equation*}
 \|(1-\chi^2)uv\|_{H_{\mathcal{M}_0}^{0,n/2;\kappa,\smo}} \le C  \|u\|_{H_{\mathcal{M}_0}^{0,0;\kappa,\smo}} \|v\|_{H_{\mathcal{M}_0}^{0,0;\kappa,\smo}}.
 \end{equation*} 

On the other hand, $u_1 \defeq \chi u, v_1 \defeq \chi v$ are supported in a collar neighbourhood of the boundary of $\partial M$, so can be viewed as belonging to $C_c^\infty([0, c) \times \partial M)$. The functions 
 $$
\tilde{u}(t,y) = u_1(e^{t}, y), \quad \tilde{v}(t,y) = v_1(e^{t}, y),
  $$
  are then defined for $(t, y ) \in \R \times \partial M$. Taking into account that $xD_x = D_t$ if $x = e^t$, we see that
\begin{equation} \label{chain rule L 2 equiv}
    \|x^{n/2}(xD_x)^jD^\beta_y w \|_{L^2([0, c)_x \times \partial M ;\, x^{-1-n}dx d\mu)} = \|D_t^j D^\beta_y \tilde{w} \|_{L^2( \R_t \times \partial M ; dt d\mu)},
\end{equation}
for all $w \in C^\infty_c([0, c) \times \partial M)$.
Using Lemma~\ref{thm:basicalgebra} and \eqref{chain rule L 2 equiv}, a short calculation shows
\begin{equation*}
\|u_1 v_1 \|_{H_{\mathcal{M}_0}^{0,n/2;\kappa,\smo}} \le C \|u_1 \|_{H_{\mathcal{M}_0}^{0,0;\kappa,\smo}} \| v_1 \|_{H_{\mathcal{M}_0}^{0,0;\kappa,\smo}},
\end{equation*}
completing the proof of Lemma \ref{thm:algebra} for the case $\sw = \sw' = s = 0$. 

One then proves the general case of $\sw, \, \sw' \in \RR$, $s \in \mathbb{N}$ using $H_{\mathcal{M}_0}^{0,\tilde{\sw};\kappa,\smo} = x^{\tilde{\sw}} H_{\mathcal{M}_0}^{0, 0;\kappa,\smo}$ (any $\tilde{\sw} \in \RR$) and the Leibniz rule. \end{proof}

The following corollary follows from Lemma~\ref{thm:algebra} and  Lemma~\ref{lem:mu-change}:

\begin{cor}\label{cor:prod} Let $s \in \mathbb{N}$ and let $\mmuu_1, \mmuu_2, \dots, \mmuu_{p+1}$ be real parameters. Then provided $\kappa \ge 1$ and $k \geq (n-1)/2$, pointwise multiplication of functions in $C^\infty_c(M)$ induces a bounded multilinear map 
\begin{equation}
 H^{s,\sw_1;\kappa, \smo}_{\mathcal{M}_{\mmuu_1}} \cdot  H^{s,\sw_2;\kappa, \smo}_{\mathcal{M}_{\mmuu_2}} \cdots H^{s,\sw_p;\kappa, \smo}_{\mathcal{M}_{\mmuu_p}} \lra
    H^{s,\sw;\kappa, \smo}_{\mathcal{M}_{\mmuu_{p+1}}}
\end{equation}
   where
\begin{equation}\label{exp}
\sw = \sw_1 + \sw_2 + \dots + \sw_p + \frac{(p-1)n}{2} - \kappa.
\end{equation} 
\end{cor}

\begin{proof}
When all $\mmuu_j = 0$, the result follows from applying
Lemma~\ref{thm:algebra} $p-1$ times, and indeed this gives a better
result without the loss of $k$ on the right side of \eqref{exp}. In general, we use 
Lemma \ref{lem:mu-change} to write an element of $H^{s, \sw_j; \kappa, \smo}_{\mathcal{M}_{\mmuu_j}}$ as
$e^{i\mmuu_j r}$ times an element $w_j$ of
$H^{s,\sw_j; \kappa, \smo}_{\mathcal{M}_{0}}$. The product is then the product
of the $w_j$, which lies in $H^{s,\sw+\kappa; \kappa, \smo}_{\mathcal{M}_{0}}$,
times the exponential factor $e^{i \mmuu' r}$ where $\mmuu' =  \mmuu_1
+ \dots + \mmuu_p$, which is in turn an element of $H^{s,\sw+\kappa;\kappa,
  \smo}_{\mathcal{M}_{\mmuu_{p+1}}}$ times the exponential $e^{i
  (\mmuu' - \mmuu_{p+1}) r}$. Finally,
multiplication by an exponential in $r$ leads to a loss of $\kappa$ in the
spatial weight since we incur a factor of $r$ each time the operator
$r (D_r - \mmuu)$ is applied to the exponential factor.
\end{proof} 

\begin{rem}
It is because of this loss of $\kappa$ in the spatial weight that we work below with spaces where $\kappa$ is as small as possible, namely $\kappa=1$. 
\end{rem}




\section{Proof of Theorem~\ref{thm:Fred prop}}\label{sec:Fredholm}

The organization of this section is as follows. We first prove invertibility of $P$ acting between variable order spaces, as in \eqref{eq:first-invertible}. We then use this to prove invertibility on the spaces with extra module regularity, as in Theorem~\ref{thm:Fred prop}. The proof of invertibility is achieved by first proving that the map in question is Fredholm, and then establishing the triviality of the kernel and cokernel. The Fredholm property is established by patching together microlocal estimates of various sorts. In the elliptic region, we use a very standard elliptic estimate; on the characteristic variety, we use a standard positive commutator estimate where the Hamilton vector field is nonvanishing, and radial-point estimates originating with Melrose at the radial sets, where the Hamilton vector field vanishes.

\begin{thm}[{{\cite[Prop. 5.28]{va18}}}]\label{variable order
    Fredholm est}
Let $\sfr_\pm \in S^{0,0}_{\emph{cl}}(\Tsc^* M)$ (that is, let them be
classical scattering symbols of order $(0,0)$ on $M$, see Section \ref{sec:scat calc}) satisfy conditions \eqref{weight-radial}, \eqref{weight-monotone} and \eqref{weight-pm}.  Let $s \in \mathbb{R}$.  Then the map
\eqref{eq:first-invertible} is invertible.
      \end{thm}

As just mentioned, the strategy is first to prove that the map \eqref{eq:first-invertible} is Fredholm. We will prove

\begin{lemma}\label{lem:variable order Fredholm}
  Suppose $\sfr_\pm \in S^{0,0}( \Tsc^* M)$ satisfying the conditions in Theorem~\ref{variable order
    Fredholm est}, let $s \in \RR$ be arbitrary, and let $M,N$ be such that $M < \min\{s, 2- s\}$ and $N < \min \{\sfr_+, \sfr_-\} = 1/2 - \delta$. Then there is a $C>0$ so that for
  all $u \in \mathcal{X}^{s, \sfr_\pm} $,
  \begin{equation}\label{eq: var ord Fredholm est}
    \norm[s,\sfr_\pm]{  u } \le C \left( \norm[s-2, \sfr_\pm + 1]{Pu} + \norm[M,N]{u} \right).
  \end{equation}
  Moreover, estimates \eqref{eq: var ord Fredholm est} for both signs $\pm$ implies that the 
map \eqref{eq:first-invertible} (for either sign) is Fredholm.
\end{lemma}

\begin{rem}
  The estimate in Lemma \ref{lem:variable order Fredholm} looks slightly
different than similar estimates one finds in the literature obtained
using radial points
estimates, e.g.\ that in \cite[Eqn. 5.2]{HVsemi}, as it does not
reflect the fact that near the higher decay region, e.g.\ near
$\mathcal{R}_-$ in the $\sfr_+$ case, one obtains estimates where the
right hand side has a term with above-threshold decay, see \eqref{eq: high-reg}
below.  This is often included in the global estimate, which can be
written
$$
  \norm[s,\sfr_+]{  u } \le C \left( \norm[s-2, \sfr_+ + 1]{Pu} +
    \norm[M,N]{u} + 
    \norm[M,\sfr']{Gu} \right) , \quad -1/2 < \sfr' < \sfr_+
  $$
and $\WF'(G)$ contained in a neighborhood of $\mathcal{R}_-$ on which
$\sfr_+$ is constant.  (See \eqref{eq:real-final-fred-est} below.) One can deduce the estimate in Lemma
\ref{lem:variable order Fredholm} from this estimate by an interpolation argument (see
  below) and we prefer \eqref{eq: var ord Fredholm est} for its simplicity.
\end{rem}

To motivate the microlocal estimates below, we briefly explain why the estimates \eqref{eq: var ord Fredholm est} imply
the Fredholm statement.  For definiteness, we consider only the $+$ sign, so we are considering
\begin{equation}
    P \colon \mathcal{X}^{s,\sfr_+} \lra
    \mathcal{Y}^{s-2,\sfr_++1}.\label{eq:var-invert-out}
      \end{equation}
The estimate in \eqref{eq: var ord Fredholm est} implies that the
kernel \eqref{eq:var-invert-out} is finite dimensional by a standard
argument. Indeed, on the
kernel of $P$ we have $\norm[s,\sfr_+]{  u } \le C \norm[M,N]{u}$, and
the containment $H^{s,\sfr_+} \subset H^{M,N}$ is compact, proving that the kernel must
be finite dimensional. 

To show that the range is closed, consider a 
sequence $u_j$ in $\mathcal{X}^{s, \sfr_+}$ in the subspace orthogonal to the kernel of $P$, for which $Pu_j$
converges to some $f \in    H_{+}^{s-2,\sfr_++1}$. Then we apply \eqref{eq: var ord Fredholm est} to $u_j$, and observe that the $\norm[M,N]{u_j}$ norms must be uniformly bounded; for if not, then one can pass to a subsequence where $\norm[M,N]{u_j}$ tends to infinity,  rescale the $u_j$ to $\hat u_j$ that  $\norm[M,N]{\hat u_j}$ is fixed to be $1$, and show, using the compact embedding $H^{s,\sfr_+} \subset H^{M,N}$ again, that a subsequence of the $\hat u_j$ converge to a limit $v$ such that $Pv = 0$. This implies that $v = 0$ since the $u_j$ were chosen orthogonal to the kernel of $P$, and this is a contradiction since it would imply convergence of the subsequence to zero also in the weaker norm $H^{M,N}$, where the norm was fixed to be $1$. 

Having thus observed that the $\norm[M,N]{u_j}$ quantities are uniformly bounded, it follows from \eqref{eq: var ord Fredholm est} that the $\norm[s,\sfr_+]{u_j}$ quantities are uniformly bounded. Using the compactness of the  inclusion $H^{s,\sfr_+} \subset H^{M,N}$ once again  we can extract a subsequence convergent in the $H^{M,N}$ norm, and applying \eqref{eq: var ord Fredholm est} we obtain convergence in the $H^{s,\sfr_+}$ norm (since the $Pu_j$ are converging in $H_{+}^{s-2,\sfr_++1}$). 
Thus we obtain a limit $u$ for this subsequence, hence $Pu_j$ converges to $Pu$. So $f = Pu$ is in the range, proving closedness of the range. 

The last step is finite dimensionality of the cokernel. The cokernel can be identified with those $v \in
(H^{s-2,\sfr_++1})^* = H^{2 - s, - 1 -\sfr_+}$ with $P v = 0$, using the formal self-adjointness of the operator $P$. Now recall from \eqref{weight-pm} that $-1 - \sfr_+$ is precisely $\sfr_-$. So it suffices to prove that the kernel of $P$ acting on $H^{2 - s, \sfr_-}$ is finite dimensional. But this follows from the estimate \eqref{eq: var ord Fredholm est} for the opposite sign $-$, exactly as above. 
This completes the proof that the estimate \eqref{eq: var ord Fredholm est}, for both signs, implies the Fredholm property of the map \eqref{eq:first-invertible} for the $+$ sign (and the argument for the $-$ works in an exactly similar manner).

\subsection{Microlocal estimates}
In this section we review the specific microlocal estimates that we shall use to prove the estimate \eqref{eq: var ord Fredholm est}, which we shall refer to as the `Fredholm estimate'. 

The first type of estimate is an elliptic estimate, that applies on the elliptic set of $P$. This is very familiar from the theory of elliptic pseudodifferential operators. The only novelty is that it applies here to the full elliptic set in the sense of the scattering calculus, thus, everywhere on the boundary of $\scTstarc M$ away from $\Sigma(P)$. 

 \begin{prop}[microlocal elliptic regularity {{\cite[Cor.\ 5.5]{va18}}} ] \label{elliptic reg}
 Let $u \in \mathcal{S}'$ and let $Q_1, G_1 \in \Psi_{\emph{sc}}^{0,0}$ be such that
 $\WF'(Q_1) \subset \Ell(G_1) \cap \Ell(P)$. Assume $G_1Pu \in H^{s-2,
   \mathsf{l}}$. Then $Q_1u \in H^{s, \sfr}$, and for all $M,N \in \R$,
 there is a constant $C > 0$ such that if $u \in H^{M,N}$, then
 \begin{equation} \label{eq: elliptic reg}
 \| Q_1u \|_{s, \mathsf{l}} \le C \left( \| G_1Pu\|_{s-2, \mathsf{l}} + \| u\|_{M,N} \right).
 \end{equation} 
 \end{prop}

Thus, on the elliptic set, $P$, as an operator of order $(2,0)$, acts microlocally  with no loss of derivatives or spatial order; $u$ is in $H^{s, \mathsf{l}}$ microlocally on the elliptic set $\Ell(P)$  if and only if $Pu$ is in 
$H^{s-2, \mathsf{l}}$ microlocally. 

On the characteristic set $\Sigma(P)$, wherever the Hamilton vector
field $H_p$ is nonvanishing, we have propagation of singularities (as
it is conventionally called --- though it is more accurately called
`propagation of regularity'). In the following proposition we
specialize to the variable order $\sfr_+$ defined in
Section \ref{Sobolev spaces of variable order}, although only the condition
\eqref{weight-monotone} is necessary for the following
result. Propagation of singularities goes back to H\"ormander's paper
\cite{Hormander:Existence}, and was first used at
spatial infinity in the scattering calculus by Melrose \cite{RBMSpec},
who viewed it as a microlocal version of the Mourre estimate
\cite{Mourre}.
The version we state is from \cite[Thm.\ 5.4]{va18}.

\begin{prop}[propagation of singularities/regularity estimate] \label{standard-prop}
   Let $u \in \mathcal{S}'$ and let $Q_2, Q_2', G_2 \in \Psi_{\emph{sc}}^{0,0}$. Assume $\WF'(Q_2) \subseteq \Ell(G_2)$. Moreover, assume that 
   \begin{equation} \label{eq: flow prop}
   \begin{gathered}
  \text{for every $\alpha \in \WF'(Q_2) \cap \Sigma(P)$, there is a  point $\alpha' \in \Ell(Q_2')$} \\
\text{and a forward bicharacteristic segment $\gamma$ from $\alpha'$ to $\alpha$ such that $\gamma  \subseteq \Ell(G_2)$.}
   \end{gathered}
   \end{equation}
    If $Q_2'u \in H^{s, \sfr_+}$ and $G_2Pu \in H^{s-2, \sfr_+ + 1}$, then $Q_2u \in H^{s,\sfr_+}$, and for all $M, N$ there is $C > 0$
   such that if $u \in H^{M,N}$, then 
   \begin{equation}
     \label{eq: standard-prop}
     \norm[s,\sfr_+]{ Q_2 u } \le C \left(\norm[s,\sfr_+]{Q_2'u} + \| G_2 P u \|_{s-2, \sfr_+ +
     1} + \norm[M,N]{u}\right).
 \end{equation}
  \end{prop}
That is, if $u$ is in $H^{s, \sfr}$ microlocally near a point $\alpha' \in \Sigma(P)$, if $\alpha$ is another point on the bicharacteristic $\gamma$ through $q$, and if $Pu$ is sufficiently regular (namely, in $H^{s-2, \sfr+1}$) along $\gamma$ between $\alpha'$ and $\alpha$, then the regularity `propagates' to $\alpha$, in the sense that  $u$ is in $H^{s, \sfr}$ at $\alpha$, \emph{provided}, in the case of a variable order $\sfr$, that $\sfr$ is nonincreasing between $\alpha'$ and $\alpha$ in the direction of bicharacteristic flow. (If a variable weight is nondecreasing in the direction of bicharacteristic flow, as is the case with $\sfr_-$, then regularity propagates in the opposite direction.) 

Neither of these estimates gives any information at the radial sets,
which are the locations within $\Sigma(P)$ where the Hamilton vector
field vanishes. At these sets, we have the following radial point
estimates, which come in two versions, one below and one above the
spatial regularity level $-1/2$ that is critical for the behaviour of solutions of $Pu = 0$. We only state these for constant spatial weight, which suffices as we have assumed that $\sfr_\pm$ are constant in a neighbourhood of the radial sets. We also have stated this proposition with $\sfr_+$ in mind, thus the below threshold result applies at the outgoing radial set $\mathcal{R}_+$, while the above threshold result applies at the incoming radial set $\mathcal{R}_-$; for the other weight $\sfr_-$, the roles of the incoming and outgoing radial sets switch.  
It will help in understanding the statements below to recall that $\mathcal{R}_+$ is a sink, and $\mathcal{R}_-$ a source, for the bicharacteristic flow, and all bicharacteristics inside $\Sigma(P)$ start at $\mathcal{R}_-$ and end at $\mathcal{R}_+$. Again, these estimates were first made by Melrose in \cite[Section 9]{RBMSpec}. 

\begin{prop}[{{\cite[Prop.\ 5.27]{va18}}}] \label{radial point prop}
\
  \begin{enumerate}
\item  Below threshold regularity radial point estimate: Assume $\sw < -1/2$.
  Let $Q_3, Q_3', G_3 \in \Psi_{\emph{sc}}^{0,0}$. Let $U, U'$ denote two open neighborhoods of $\mathcal{R}_+$ with $U \Subset U' \Subset  \Tsc^*_{\p M} M$, and assume $U \subset \Ell(Q_3) \subset \WF'(Q_3) \subset  \Ell(G_3) \subset U'$. Assume that $\WF'(Q_3')$ is contained in $U' \setminus U$ and that, 
   \begin{equation} \label{eq: flow prop 2}
   \begin{gathered}
  \text{for every $\alpha \in \WF'(Q_3) \cap (\Sigma(P) \setminus \mathcal{R}_+)$, there is a  point $\alpha' \in \Ell(Q_3')$} \\
\text{and a forward bicharacteristic segment $\gamma$ from $\alpha'$ to $\alpha$ such that $\gamma  \subseteq \Ell(G_3)$.}
   \end{gathered}
   \end{equation}
    If $Q_3'u \in H^{s, \sw}$ and $G_3Pu \in H^{s-2, \sw + 1}$, then $Q_3u \in H^{s,\sw}$, and for all $M, N$ there is $C > 0$
   such that if $u \in H^{M,N}$, then 
   \begin{equation}
     \label{eq: low-reg}
     \norm[s,\sw]{ Q_3 u } \le C \left(\norm[s,\sw]{Q_3'u} + \| G_3 P u \|_{s-2, \sw +
     1} + \norm[M,N]{u}\right).
  \end{equation}   

  \item Above threshold regularity: Assume $\sw,\sw' >
    -1/2$ and $s, s' \in \mathbb{R}$. Let $ U_- \Subset \scTstar M$ be a sufficiently small neighborhood of $\mathcal{R}_-$. Then for all $Q_4, G_4 \in \Psi_{\emph{sc}}^{0,0}$ such that 
    $$
    \mathcal{R}_- \subset \Ell(Q_4) \subset \WF'(Q_4) \subset \Ell(G_4) \subset  U_-, 
    $$
 if  $G_4 Pu \in H^{s-2, \sw+1}$ and $G_4 u \in H^{s',\sw'}$, then $Q_4u \in H^{s,\sw}$. Moreover, for all $M,N$, there is $C > 0$ so that if $u \in H^{M,N}$, then
    \begin{equation}
      \label{eq: high-reg}
      \norm[s,\sw]{Q_4u} \le C\left(  \norm[s',  \sw']{G_4 u} + \norm[s-2, \sw +
1]{G_4 Pu}  + \norm[M,N]{u}\right).
    \end{equation}

 \end{enumerate}
\end{prop}

\subsection{Fredholm estimate}\label{subsec:global}
In this subsection we explain how to piece together the microlocal
estimates to produce a global estimate. See also Vasy,
\cite[Sect. 5.4.6]{va18} where this piecing together of estimates is
discussed in terms of wavefront sets.

We first note that, if $U'$ and $U_-$ are chosen small enough such that $\sfr_+$ is constant, equal to $-1/2 - \delta$ near $U'$ and $-1/2 + \delta$ near $U_-$, and if we choose $r$ to be equal to these respective values in \eqref{eq: low-reg} and \eqref{eq: high-reg}, then we can deduce the estimates with the variable weight $\sfr_+$, as follows: 
\begin{equation}
     \label{eq: low-reg 2}
     \norm[s,\sfr_+]{ Q_3 u } \le C \left(\norm[s,\sfr_+]{Q_3'u} + \| G_3 P u \|_{s-2, \sfr_+ +
     1} + \norm[M,N]{u}\right)
  \end{equation}   
  and 
 \begin{equation}
      \label{eq: high-reg 2}
      \norm[s,\sfr_+]{Q_4u} \le C\left(  \norm[s',  r']{G_4 u} + \norm[s-2, \sfr_+ +
1]{G_4 Pu}  + \norm[M,N]{u}\right).
    \end{equation}  
This is because the `microlocal difference' between, say, the norms $\norm[s,-1/2 - \delta]{ Q_3 u }$ and $\norm[s,\sfr_+]{ Q_3 u }$ is disjoint from the microlocal support of $Q_3$, so the difference can be controlled by $\norm[M,N]{u}$ for arbitrary $M$ and $N$; exactly the same argument applies to each of the other terms in these two estimates.

We then combine the estimates \eqref{eq: elliptic reg}, \eqref{eq: standard-prop}, \eqref{eq: low-reg 2} and \eqref{eq: high-reg 2}, noting that we may assume that $Q_1 + Q_2 + Q_3 + Q_4 = \Id$.  In addition, we may assume that $Q_2' = Q_4$ and $Q_3' = Q_2$, as these satisfy the propagation conditions in \eqref{eq: flow prop} and \eqref{eq: flow prop 2}. We add up these estimates, building in a large multiple of the second estimate and an even larger multiple of the fourth. Thus, for a constant $C$ equal to the maximum constant in the four estimates \eqref{eq: elliptic reg}, \eqref{eq: standard-prop}, \eqref{eq: low-reg} and \eqref{eq: high-reg}, we have 
\begin{equation}\begin{aligned}
\| u \|_{s, \sfr_+} &\leq \| Q_1u \|_{s, \sfr_+} + K  \| Q_2 u \|_{s, \sfr_+} + \| Q_3u \|_{s, \sfr_+} + K^2 \| Q_4u \|_{s, \sfr_+} \\
\leq C &\bigg( \| G_1Pu\|_{s-2, \mathsf{l}_+} + \| u\|_{M,N}  + K \Big( \norm[s,\sfr_+]{Q_4u} +  \| G_2 P u \|_{s-2, \sfr_+ + 1} +  \norm[M,N]{u} \Big) \\
     + \norm[s,\sfr_+]{Q_2u} + &\| G_3 P u \|_{s-2, \sfr_+ +1} + \norm[M,N]{u} + K^2 \Big( \norm[s',  \sw']{G_4 u} +  \norm[s-2, \sfr_+ +1]{G_4 Pu}  +  \norm[M,N]{u} \Big) \bigg)
\end{aligned}\end{equation}

Next, we estimate the $Q_2$ and $Q_4$ terms on
the RHS as just done, by using \eqref{eq: standard-prop}, respectively \eqref{eq: high-reg 2}. Notice that estimating the $Q_2$ term produces an additional $Q_4$ term, we again estimate using \eqref{eq: high-reg 2}. So, finally, we obtain, for a new larger constant $C$ (noting that $s'$ is arbitrary),
\begin{equation}\label{eq:real-final-fred-est}
\| u \|_{s, \sfr_+} \le C \left( \| Pu\|_{s-2, \mathsf{r}_+ + 1} + \| u\|_{M,N} + \norm[s',  r']{G_4 u} \right).
\end{equation}

Now, to handle the $G_4$ term, choose $\sw' = -1/2 + \delta/2$ and $M < s' < s$. By Sobolev space interpolation, for appropriate $\eta \in (0,1)$,
\begin{equation*}
\|G_4 u \|_{s', -1/2 + \delta/2} \le   \|G_4 u \|^{1 - \eta}_{s, -1/2 + \delta} \|G_4 u \|^\eta_{M, N}.
\end{equation*}
 We can replace the norm $\norm[s, -1/2 + \delta]{G_4 u}$ with 
$\norm[s',  \mathsf{l}_+]{G_4 u}$ since $\sfr_+ = -1/2 + \delta$ on the microlocal support of $G_4$ (which is contained in $U_-$). Then, by Young's inequality,
$$
 \norm[s,  \sfr_+ ]{G_4 u}^{1 - \eta} \norm[M, N]{G_4 u}^\eta \leq 
\epsilon \norm[s,  \sfr_+ ]{u} + C(\epsilon) \norm[M, N]{u}
$$
 for arbitrary $\epsilon > 0$  provided $C(\epsilon)$ is sufficiently large. 
The $ \epsilon \norm[s,  \sfr_+ ]{u}$ term can be absorbed into the RHS of \eqref{eq:real-final-fred-est} and the other term is a multiple of $\norm[M, N]{u}$. This yields the Fredholm estimate \eqref{eq: var ord Fredholm est}. 

\subsection{Invertibility on variable order spaces}
We now prove Theorem~\ref{variable order Fredholm est} using the result of Lemma~\ref{lem:variable order Fredholm}.
Given that the map \eqref{eq:first-invertible} is Fredholm, it only remains to show that the kernel and cokernel are both trivial. In fact, due to the formal self-adjointness of the operator $P$, this amounts to showing that if $Pu = 0$, and either $u \in H^{s,\sfr_+}$ or $H^{s, \sfr_-}$, then $u=0$. As the argument is essentially the same in both cases, we only consider the case that $u \in H^{s,\sfr_+}$.

So, assume that $u \in H^{s,\sfr_+}$ and $Pu = 0$. Then, $u$ is in
$H^{s, -1/2 + \delta}$ microlocally in a neighbourhood of
$\mathcal{R}_-$, so we can apply Proposition~\ref{radial point prop},
part (ii), and deduce that $u$ is in $H^{s, L}$ for arbitrarily large
$L$ microlocally near $\mathcal{R}_-$. The propagation theorem, 
Proposition~\ref{standard-prop}, then shows that $u$ is in $H^{s, L}$
for arbitrary $L$ everywhere on $\Sigma$ except, possibly, at
$\mathcal{R}_+$. The elliptic estimate implies that in fact, $u$ is
microlocally trivial except possibly at $\mathcal{R}_+$, in the sense
that if $A$ is such that $\WF'(A)$ is disjoint from $\mathcal{R}_+$,
then $Au$ is smooth on $M$ and all derivatives rapidly vanishing at the boundary (in particular, $O(r^{-N})$ for every $N$). 

We can thus apply \cite[Prop.\ 12]{RBMSpec}, which tells us that if $u$  
has wavefront set contained in $\mathcal{R}_+$, and $Pu \in \mathcal{S}$, then $u$ has the form 
$$
u = r^{-(n-1)/2} e^{i\lambda r} \sum_{j=0}^\infty r^{-j} v_j(y), \quad r \to \infty, \text{ where } v_j \in C^\infty(\partial M). 
$$
On the other hand, the ``boundary pairing'' lemma \cite[Prop.\ 13]{RBMSpec}
shows that the leading coefficient $v_0$ in the expansion of $u$ satisfies 
$$
- 2 i \lambda \int_{\p M} |v_0|^2 = 2 \re \int_{M}  u \overline{Pu } .
$$
Since the right hand side is zero, $v_0 \equiv 0$ and thus $u \in H^{\infty, - 1/2 +
  \delta}(M)$ for $\delta$ small enough. Thus, $u$ is above threshold decay at
both radial sets, and again using Proposition~\ref{radial point prop},
part (ii) (at the outgoing rather than incoming radial set) and $Pu =
0$, it follows that $u$ vanishes to infinity order along with its
derivatives, i.e.\ $u \in  \dot C^\infty(M) $.
Finally, we can apply  \cite[Theorem 17.2.8]{Hor}, or alternatively \cite{FHOO}, to deduce that $u
\equiv 0$. This completes the proof of Theorem~\ref{variable order Fredholm est}.


\begin{rem}
The fact that $P^{-1}$ on these variable order spaces is
equal to the action of the outgoing resolvent is shown in \cite[Sect.\ 11]{RBMSpec}, or \cite{Vasy-LAP}. 
 \end{rem}

\subsection{Module regularity}
The next step is to adapt the argument above to the module regularity spaces $H_+^{s,\sw;\kappa, \smo}$ instead of variable order spaces $H^{s, \sfr_+}$. We shall prove each of the microlocal estimates above in the module regularity setting. 

 \begin{prop}[microlocal elliptic regularity --- module version] \label{elliptic reg mod}
 Let $u \in \mathcal{S}'$ and let $Q_1, G_1 \in \Psi_{\emph{sc}}^{0,0}$ be such that
 $\WF'(Q_1) \subset \Ell(G_1) \cap \Ell(P)$. Assume $G_1Pu \in H_+^{s-2,
  \sw; \kappa,\smo}$. Then $Q_1u \in H_+^{s, \sw; \kappa,\smo}$, and for all $M,N \in \R$,
 there is a constant $C > 0$ such that if $u \in H^{M,N}$, then
 \begin{equation} \label{eq: elliptic reg mod}
 \| Q_1u \|_{s, \sw; \kappa, \smo} \le C \left( \| G_1Pu\|_{s-2, \sw; \kappa,\smo}  + \| u\|_{M,N} \right).
 \end{equation} 
 \end{prop}

\begin{proof} We prove this by induction on $(\kappa, k)$. For $(\kappa,k) = (0,0)$ this is just Proposition~\ref{elliptic reg}.
Now assume, for a given $(\kappa, k)$ that the result is true for all
$(\kappa', k') < (\kappa, k)$ in the sense that $\kappa' \leq \kappa$,
$k'\leq k$ and $(\kappa',k') \neq (\kappa, k)$. Then, generators $A_1, \dots, A_m$ of $\mathcal{M}_+$ and $B_1, \dots, B_l$ of $\mathcal{N}$, we have 
\begin{multline}
A_1 \cdots A_m B_1 \cdots B_l Q_1 u = Q_1 A_1 \cdots A_m B_1 \cdots B_l u \\ + \sum_{j=1}^m A_1 \cdots [A_j, Q_1] \cdots A_m B_1 \cdots B_l u+ \sum_{j=1}^l A_1 \cdots A_m B_1 \cdots [B_j, Q_1] \cdots B_l u .
\end{multline}
We can shift the commutator factors to the left of the product modulo
double commutator factors, shift the  double commutator factors to the
left modulo triple commutator factors and so on. Notice that all of
these multiple commutator factors are order $(0,0)$, with microlocal
support no bigger than $\WF'(Q_1)$, hence contained in the elliptic
set of $G_1$. Note first that $G_1 P A_1 \dots A_m B_1 \dots B_lu \in
H^{s-2, \sw}.$  Indeed this can be seen by writing the operator in
terms of commutators as
\begin{multline}
G_1 P A_1 \cdots A_m B_1 \cdots B_l = A_1 \cdots A_k B_1 \cdots B_l
G_1 P \\ + \sum_{j=1}^m A_1 \cdots [G_1 P, A_j] \cdots A_m B_1 \cdots
B_l + \sum_{j=1}^l A_1 \cdots A_m B_1 \cdots [G_1 P, B_j] \cdots B_l,
\end{multline}
so using that $[G_1 P, A_j],  [G_1 P, B_j] \in \Psi^{2, 0}$, we can therefore apply 
Proposition~\ref{elliptic reg} to obtain 
\begin{equation}\label{comm-elliptic}
  \begin{gathered}
    \| A_1 \cdots A_m B_1 \cdots B_l Q_1 u \|_{s,\sw} \lesssim \| G_1
    P A_1 \cdots A_m B_1 \cdots B_l u \|_{s-2, \sw} \\
    + \text{
      commutator terms} + \| u \|_{M, N}.
  \end{gathered}
\end{equation}
and we can perform a similar process as above, shifting the commutator factors to the left modulo double commutator factors, shifting those to the left modulo triple commutator factors, and so on. Each of these multiple commutator factors are order $(2,0)$. Substituting into \eqref{comm-elliptic} we obtain
\begin{equation}
\begin{split}
\| A_1 \cdots A_m B_1 \cdots& B_l Q_1 u \|_{s,\sw} \lesssim \| A_1 \cdots
A_m B_1 \cdots B_l G_1 P u \|_{s-2, \sw} \\&+ \sum \| \tilde{C}_\bullet \prod
A_\bullet \prod
B_\bullet u \|_{s, \sw}+ \sum \| C_\bullet \prod
A_\bullet \prod
B_\bullet u \|_{s-2, \sw} + \|u\|_{M,N},
\end{split}
\end{equation}
where $C_\bullet \in \Psi_{\text{sc}}^{2,0}$ and $\tilde{C}_\bullet
\in \Psi_{\text{sc}}^{0,0}$ are multi-commutators with wavefront set
contained in $\WF'(Q_1)$, and we have fewer than $\kappa+k$ factors of the $A_\bullet$ and the $B_\bullet$ in total. The other terms in \eqref{comm-elliptic} are estimated similarly. We thus obtain 
\begin{equation}
\| A_1 \cdots A_m B_1 \cdots B_l Q_1 u \|_{s,\sw} \lesssim \| G_1 P u \|_{s-2, \sw; \kappa, \smo} + \sum_{(\kappa',k') < (\kappa,k)} \| G' u \|_{s, \sw; \kappa', k'} + \| u \|_{M, N},
\end{equation}
where $G'$ is chosen so that $\WF'(Q_1) \subset \Ell(G') \subset
\WF'(G') \subset \Ell(G_1)$ and $\WF(I - G')  \cap \WF'(Q_1) = \varnothing$. These conditions imply that $G' C = C$
modulo an operator of order $(-\infty, -\infty)$ which contributes to the $ \| u \|_{M, N}$ term.  We apply the inductive assumption to the term $G'u$, where $G'$ now plays the role of $Q_1$, and arrive at 
\begin{equation}
\| A_1 \cdots A_k B_1 \cdots B_l Q_1 u \|_{s,\sw} \lesssim \| G_1 P u \|_{s-2, \sw; \kappa, \smo}  + \| u \|_{M, N}. 
\end{equation}
After summing over all possible choices of the $A_1 \dots A_m$ and the
$ B_1 \dots B_l$ we obtain \eqref{eq: elliptic reg mod}.  \end{proof}

\begin{prop}[propagation of regularity estimate --- module version] \label{standard-prop mod}
   Let $u \in \mathcal{S}'$ and let $Q_2, Q_2', G_2 \in \Psi_{\emph{sc}}^{0,0}$. Assume $\WF'(Q_2) \subseteq \Ell(G_2) \setminus (\mathcal{R}_+ \cup \mathcal{R}_-)$. Moreover, assume that \eqref{eq: flow prop} holds. 
    If $Q_2'u \in H_+^{s, \sw; \kappa,\smo}$ and $G_2Pu \in H_+^{s-2, \sw+ 1; \kappa,\smo}$, then $Q_2u \in H_+^{s,\sw; \kappa,\smo}$, and for all $M, N$ there is $C > 0$
   such that if $u \in H^{M,N}$, then 
   \begin{equation}
     \label{eq: standard-prop mod}
     \norm[s,\sw; \kappa,k]{ Q_2 u } \le C \left(\norm[s,\sw; \kappa,k]{Q_2'u} + \| G_2 P u \|_{s-2, \sw +1; \kappa,\smo} + \norm[M,N]{u}\right).
 \end{equation}
  \end{prop}
  
\begin{rem}  
Here we included the extra assumption that $\WF'(Q_2)$ is disjoint from the radial sets, so that the modules $\mathcal{M}_+$ and $\mathcal{N}$ both become \emph{elliptic} on $\WF'(Q_2)$, in the sense that at each point of $\WF'(Q_2)$ there exists a module element of  $\mathcal{N}$ (and hence also $\mathcal{M}_+$) that is elliptic. (It was not necessary to include this assumption in Proposition~\ref{standard-prop}, but we could have done so as the Proposition gives no information at the radial sets.) 
\end{rem}
\begin{proof}
When the modules are elliptic the proof becomes almost trivial. We note that the module norm $\| \cdot \|_{s,\sw; \kappa, \smo}$ is equivalent, microlocally on $\WF'(Q_2)$, to the $\| \cdot \|_{s+\kappa+k,\sw+\kappa+\smo}$ norm. So the Proposition is actually equivalent to the previous one, with a shift in orders $s$ and $\sw$ by $\kappa+k$. 
\end{proof}

In the next proposition, all operators have microlocal support in a compact region of $\scTstar M$ by assumption, thus disjoint from fibre-infinity. Hence the differential order is irrelevant for both the operators and the spaces. We write it $*$ to emphasize this irrelevance. 

\begin{prop}[Radial point estimates --- module version] \label{radial point prop mod}
  \
  \begin{enumerate}
\item  Below threshold regularity radial point estimate: Assume $\sw < -1/2$.
  Let $Q_3, Q_3', G_3 \in \Psi_{\emph{sc}}^{0,0}$. Let $U, U'$ denote two open neighborhoods of $\mathcal{R}_+$ with $U \Subset U' \Subset  \Tsc^*_{\p M} M$, and assume that $U \subset \Ell(Q_3) \subset \WF'(Q_3) \subset  \Ell(G_3) \subset U'$. Assume that $\WF'(Q_3')$ is contained in $U' \setminus U$ and that \eqref{eq: flow prop 2} holds. 
    If $Q_3'u \in H_+^{*, \sw; \kappa,\smo}$ and $G_3Pu \in H_+^{*, \sw + 1; \kappa,\smo}$, then $Q_3u \in H_+^{*,\sw; \kappa,\smo}$, and for all $M, N$ there is $C > 0$
   such that if $u \in H^{M,N}$, then
   \begin{equation}
     \label{eq: low-reg mod}
     \norm[*,\sw; \kappa,k]{ Q_3 u } \le C \left(\norm[*,\sw; \kappa,k]{Q_3'u} + \| G_3 P u \|_{*, \sw +
     1; \kappa,\smo} + \norm[M,N]{u}\right).
  \end{equation}   

  \item Above threshold regularity: Assume $\sw,\sw' >
    -1/2$. Let $ U_- \Subset \scTstar_{\partial M} M$ be a sufficiently small neighborhood of $\mathcal{R}_-$. Then for all $Q_4, G_4 \in \Psi_{\emph{sc}}^{0,0}$ such that 
    $$
    \mathcal{R}_- \subset \Ell(Q_4) \subset \WF'(Q_4) \subset \Ell(G_4) \subset  U_-, 
    $$
 if  $G_4 Pu \in H_+^{*, \sw+1; \kappa,\smo}$ and $G_4 u \in H_+^{*,\sw'; \kappa,\smo}$, then $Q_4u \in H_+^{*,\sw; \kappa,\smo}$. Moreover, for all $M,N$, there is $C > 0$ so that if $u \in H^{M,N}$, then
    \begin{equation}
      \label{eq: high-reg mod}
      \norm[*,\sw; \kappa,k]{Q_4u} \le C\left(  \norm[*,  \sw'; \kappa,k]{G_4 u} + \norm[*, \sw +
1; \kappa,k]{G_4 Pu}  + \norm[M,N]{u}\right).
    \end{equation}
 \end{enumerate}

Moreover, all of the above holds with $H_+^{*, \sw; \kappa,\smo},
\mathcal{R}_\pm$ replaced by $H_-^{*, \sw; \kappa,\smo}, \mathcal{R}_{\mp}$.
\end{prop}

\begin{proof}
In this case, the argument is more elaborate than the previous two proofs, and relies on the construction of a positive commutator. The key fact we use is Lemma~\ref{thm:posititivy-criticality}, that is, the $P$-positivity (in fact $P$-criticality) of module $\mathcal{N}$.  It is very similar to the argument from 
\cite[Section 6]{HMV2004}, where test modules were introduced. To avoid a long exposition about test modules and how positive commutator estimates are used to prove module regularity, we will use Section 6 of \cite{HMV2004} as a basis and only indicate the minor differences that arise in the present case. 

We first prove \eqref{eq: low-reg mod} for $\kappa = 0$, that is, when only the module $\mathcal{N}$ is involved. We fix  a basis $A_0 = \Id, \dots, A_N$ of the module $\mathcal{N}$, and use the notation $A_{\alpha}$, $\alpha = (\alpha_1, \dots, \alpha_N)$ a multi-index, for the operator 
$$
A_1^{\alpha_1} \dots A_N^{\alpha_N}
$$
and $A_{\alpha, \sw}$ for the operator $x^{-\sw} A_\alpha$. 
We note that the $A_\alpha$, as $\alpha$ ranges over all multi-indices of length $m$, together with $\Id$ forms a basis for $\mathcal{N}^m$, the vector space of sums of $m$-fold products of elements of $\mathcal{N}$, as a module over $\Psi_{\text{sc}}^{0,0}$. 

We prove the estimate by induction on $k$, the module order. For $k=0$ the result is precisely Proposition~\ref{radial point prop}. We now assume inductively that result has been proved for all $k' < k$. 
The positive commutator estimate arises from the following operator identity, which is equation (6.16) in \cite{HMV2004}. In the following, $Q$ is arbitrary, but we will choose it to be an operator which is microlocally equal to the identity on $\WF'(Q_3)$, and with $\WF'(Q) \subset \Ell(G_3)$. In the following identity, the $C_{jk}$ are defined by the commutators of $P$ with basis elements $A_j$, as in \eqref{off diag symbols vanish}. 
\begin{equation}
  \label{eq:23}
  \begin{gathered}
    i [A_{\alpha, \sw + 1/2}^* Q^* Q A_{\alpha, \sw + 1/2}, P]
    = A_{\alpha, \sw} Q^* \left( C_0  +  C_0^* + \sum_{j = 1}^{N}
      \alpha_j (C_{jj} + C_{jj}^*) \right) Q A_{\alpha, \sw} \\
    + \sum_{|\beta| = k, \beta \neq \alpha} A_{\alpha, \sw} Q^* C_{\alpha \beta} Q A_{\beta, \sw} +
    \sum_{|\beta| = k, \beta \neq \alpha} A_{\beta, \sw} Q^* C_{\alpha \beta}^* Q A_{\alpha, \sw}  \\ +
    A_{\alpha, \sw}^* Q^* E_{\alpha, \sw} 
   +  E_{\alpha, \sw}^* Q A_{\alpha, \sw} +   A_{\alpha, \sw + 1/2}^* i [Q^* Q , P] A_{\alpha, \sw + 1/2}
  \end{gathered}
\end{equation}
where
\begin{equation}
  \label{eq:24}
  \begin{split}
    \sigma_{\text{base},0}(C_0) \rvert_{\mathcal{R}_+} &=  -\lambda (2 \sw + 1)   \\
    \sigma_{\text{base},0}(C_{\alpha \beta}) \rvert_{\mathcal{R}_+} &= 0, \  C_{\alpha \beta} \in \Psi_{\text{sc}}^{1,0}(M),  \\
    E_{\alpha, \sw} &= x^{-\sw} E_\alpha, \quad E_\alpha \in
    \mathcal{N}^{k-1} . 
  \end{split}
\end{equation}
The key point above is that the symbol of $C_0$, arising from the $2\nu x \partial_x$ component of minus the Hamilton vector field from \eqref{ac-Hvf} hitting the $x^{-\sw}$ factor,  has a definite sign near $\mathcal{R}_+$ --- positive for $\sw$ less than the threshold exponent $-1/2$. Moreover,  the $P$-criticality of $\mathcal{N}$ means that the diagonal operators $C_{jj}$ have symbols vanishing at  $\mathcal{R}_+$ --- cf. \eqref{off diag symbols vanish}.  Similarly, the off-diagonal terms $C_{\alpha \beta}$ vanish at $\mathcal{R}_+$ due to \eqref{off diag symbols vanish}. Now, we define a matrix $C' = (C'_{\alpha \beta})$ of operators, as the indices $\alpha$, $\beta$ vary over multi-indices of length $k$, as follows: for $\alpha \neq \beta$, 
$$
C'_{\alpha \beta} = C_{\alpha \beta} + C_{\beta \alpha}^*
$$ 
 and on the diagonal, we define 
 \begin{equation}
C'_{\alpha \alpha} = C_0  +  C_0^* + \sum_{j = 1}^{N}
      \alpha_j (C_{jj} + C_{jj}^*).\label{eq:C'alphaalpha}
          \end{equation}
 Thus, due to \eqref{eq:24}, the symbol of $C'$ at $\mathcal{R}_+$ is 
 diagonal with positive entries, and it is therefore 
 positive as an matrix, provided that  the microlocal support of $Q$ is sufficiently close to $\mathcal{R}_+$. This means that we can write 
 $$
 Q ^* C' Q = Q^* (B^* B + G) Q, 
 $$
 where $B$ is a matrix of operators of order $(*,0)$ and $G$ a matrix
 of operators of order $(*, -1)$. For a function $u$ we also write
 $Au$ for $(Q A_{\alpha, \sw} u)$, regarded as a column vector
 indexed by multi-indices $\alpha$ of length $m$.
 Thus, in this compact notation we can write the first two lines on
 the RHS of  \eqref{eq:23} as $A^* (B^* B + G) A$.

 Now we follow the argument of the proof of \cite[Proposition 6.7]{HMV2004}. We let $u'$ be an element of $H_+^{*,\sw; \kappa, 0}$. We have, in matrix notation, 
 \begin{equation}\begin{gathered}
 \sum_{|\alpha| = k} \big\langle u', i [A^*_{\alpha, \sw+1/2} Q^* Q A_{\alpha, \sw+1/2}, P] u' \big\rangle \\ 
 = \| BAu' \|^2 + \ang{Au', G Au'} + \sum_{|\alpha| = k} \big( \ang{Q A_{\alpha, \sw} u', E_{\alpha, \sw} u'} + \ang{E_{\alpha, \sw} u', Q A_{\alpha, \sw} u'} \big) \\
 + \sum_{|\alpha| = k} \ang{A_{\alpha, \sw} u', F A_{\alpha, \sw} u'}. 
 \end{gathered}\end{equation}
 Here $F = [Q^* Q, P]$ has operator wavefront set in $U' \setminus U$,
 in particular, disjoint from $\mathcal{R}_+$, and for elements which
 are understood to be vectors of distributions, the inner product is
 the direct sum inner product.
 
 We may assume that there is a $\tilde Q$ microlocally equal to the identity on $\WF'(Q)$, and with $\WF'(\tilde Q) \subset \Ell(G)_3)$. We can therefore write 
 $$
 \ang{A_{\alpha, \sw} u', F A_{\alpha, \sw} u'}  =  \ang{\tilde QA_{\alpha, \sw} u', F A_{\alpha, \sw} u'} + \ang{Eu',  A_{\alpha, \sw} u'}
 $$
 where $E$ is order $(-\infty, -\infty)$. Making this substitution,
 rearranging and applying the Cauchy-Schwarz inequality, followed by
 the inequality $ab \leq \epsilon a^2 + \epsilon^{-1} b^2$, we obtain,
 for some $C$ independent of $u'$, and all norms understood to be
 $L^2$-norms unless otherwise stated,
\begin{equation}\label{eq:BA}\begin{gathered}
\| BAu' \|^2 \leq \sum_\alpha \Big| \ang{Q A_{\alpha, \sw+1/2} u', Q A_{\alpha, \sw+1/2} Pu'} \Big|  \\
+  \epsilon   C \Big(  \| Au' \|^2 + \sum_\alpha \| \tilde Q A_{\alpha, \sw} u' \|^2 \Big) \\
+ \epsilon^{-1} \Big( \| G Au' \|^2 + 2 \sum_\alpha \big( \| E_{\alpha, \sw} u' \|^2 + \|  F A_{\alpha, \sw} u' \|^2 \big) \Big) + C \| u' \|_{M,N}^2. 
\end{gathered}\end{equation}
We can treat the commutator term similarly (this is not done in \cite{HMV2004}, since there it was assumed that $Pu'$ is Schwartz). Notice that $Q - x^{-1/2} Q x^{1/2}$ is an operator of order $(0, -1)$. Combined with $A_{\alpha, \sw+1/2}$ this gives us an element of the $(k-1)$-th power of the module $\mathcal{M}_+$, which we shall write (abusing notation somewhat) as $E_{\alpha, \sw+1/2}$. Then we have 
$$
\ang{Q A_{\alpha, \sw+1/2} u', Q A_{\alpha, \sw+1/2} Pu'} = \ang{Q A_{\alpha, \sw} u', x^{-1/2} Q A_{\alpha, \sw+1/2} Pu'} + \ang{E_{\alpha, \sw} u', x^{-1/2} Q A_{\alpha, \sw+1/2} Pu'}
$$
and therefore 
$$
\Big| \ang{Q A_{\alpha, \sw+1/2} u', Q A_{\alpha, \sw+1/2} Pu'} \Big| \leq  \epsilon \| Q A_{\alpha, \sw} u' \|^2  +   \| E_{\alpha, \sw} u' \|^2 + (1 + \epsilon^{-1}) \| x^{-1/2} Q A_{\alpha, \sw+1/2} Pu' \|^2. 
$$
Summing over $\alpha$ and combining this with \eqref{eq:BA} we have 
\begin{equation}\label{eq:BA2}\begin{gathered}
\| BAu' \|^2 \leq  \epsilon \Big( 4\| Au' \|^2 + \sum_\alpha \| \tilde Q A_{\alpha, \sw} u' \|^2 \Big) \\
+ \epsilon^{-1} \Big( \| G Au' \|^2 + 3 \sum_\alpha \big( \| E_{\alpha, \sw} u' \|^2 + \|  F A_{\alpha, \sw} u' \|^2  +2 \| x^{-1/2} Q A_{\alpha, \sw+1/2} Pu' \|^2\big) \Big) + C \| u' \|_{M,N}^2. 
\end{gathered}\end{equation}
The terms proportional to $\epsilon$ can be absorbed in the LHS, up to a term of the form $C \| u' \|_{M,N}^2$. In fact, on the microlocal support of $Q$, $B$ has a microlocal inverse, that we will denote $B^{-1}$ (despite not being an actual inverse of $B$). So we have $A = B^{-1} B A + E'$, where $E'$ has order $(-\infty, -\infty)$. Then, estimating $B^{-1}$ by its operator norm, we can absorb the $\| A u' \|^2$ terms provided $\epsilon$ is small compared to $\| B^{-1} \|$, while the $E'$ term only contributes a multiple of $\|u' \|_{M,N}^2$. 

We now notice that $G A$ can be treated as being in the  $(k-1)$-th power of the module since $G$ has order $(*,-1)$. So this term, as well as the $E_{\alpha, \sw}$ term, can be estimated using the inductive assumption. Similarly, we can commute the $F$ factor to the right of the $A_{\alpha, \sw}$ (up to terms in the $(k-1)$-th power of the module) and then replace it by $Q_3'$ since it can be written $F  = F' Q_3' + E''$ for some $E''$ of spatial order $-\infty$. In exactly the same way, we can commute the $Q$ to the right of the $A_{\alpha, \sw+1/2}$ factor and then replace it with $G_3$. Similarly, on the LHS of \eqref{eq:BA2},  $Q$ can be moved to the right of the $A_{\alpha, \sw}$, and then $B$ can be removed just as for the $Au'$ term above. Moreover, as $Q$ is microlocally equal to the identity on $\WF'(Q_3)$, we can replace it with $Q_3$ on the LHS. 
After these manipulations, we obtain the estimate 
\begin{equation}\label{BAest}
\| Q_3u' \|^2_{*, \sw; 0, \smo} \leq  C \Big( \| G_3 Pu' \|_{*, \sw+1;
  0, \smo}^2 +   \| Q_3' u' \|_{*, \sw; 0 , \smo}^2 
+  \| u' \|_{M,N}^2 \Big) .
\end{equation}
Now we let $u' = u'(\eta) := (1 + \eta r)^{-1}u $, $u \in H_+^{*, \sw;
  0, \smo - 1}$, for $\eta > 0$ tending to zero. Then $u' \in H_+^{*,
  \sw; 0, \smo}$ for each $\eta > 0$, so the above computation is valid. Assuming that 
$Q_3'u \in H_+^{*, \sw; 0, \smo}$ and $G_3Pu \in H_+^{*, \sw + 1; 0, \smo}$, then the RHS of \eqref{BAest} stays bounded as $\eta \to 0$. Therefore, the LHS also stays bounded, and using the strong convergence of $(1 + \eta r)^{-1}$ to the identity as in \cite[Lemma 4.3]{HMV2004}, we see that we obtain estimate \eqref{BAest} also with $u' = u$. 

Next, we shall show a slight strengthening of \eqref{eq: low-reg mod}  for $\kappa = 0$: we shall show that 
 \begin{equation}
     \label{eq: low-reg mod-stronger-0}
     \norm[*,\sw; 1, \smo-1]{ Q_3 u } \le C \left(\norm[*,\sw; 0,\smo]{Q_3'u} + \| G_3 P u \|_{*, \sw +
     1; 0,\smo} + \norm[M,N]{u}\right), \quad \smo \geq 1.
  \end{equation}   

To show this, notice that we have already shown that $Q_3 u$ is in $H^{*, \sw; 0, \smo}_+$ with the required estimate. So it only remains to prove an estimate for the additional element $A_+ := r(D_r - \lambda)$ that is in $\mathcal{M}_+$ but not in $\mathcal{N}$. To do this, 
we write $Pu$ in the form 
$$
Pu = (D_r + \lambda) (D_r - \lambda)u + i(n-1) \frac{D_r u}{r} + r^{-2} \Omega u , 
$$
where $\Omega$ involves only tangential differentiation of order at most two, with coefficients smooth on $M$. By assumption, this is in $H^{*, \sw + 1; 0, \smo}_+$ microlocally on $\WF'(Q_3)$. We rearrange as  
$$
(D_r + \lambda) r^{-1} A_+ u = Pu -  i(n-1) \frac{D_r u}{r} - r^{-2} \Omega u , 
$$
and apply Lemma~\ref{lem:psdo-mapping}. 
On the RHS, notice that $Pu$ by assumption is in $H^{*, \sw+1; 0, \smo}_+$ microlocally on $\WF'(Q_3)$;  that $D_r u /r$ is in $H^{*, \sw+1; 0, \smo}_+$ microlocally on $\WF'(Q_3)$; and finally that $r^{-2} \Omega u$ is in 
$H^{*, \sw+1; 0, \smo-1}_+$ microlocally on $\WF'(Q_3)$ (by viewing one of the $\partial_{y_i}$ factors in $\Omega$ as being in the module $\mathcal{N}$, and a second factor, $r^{-1} \partial_{y_j}$, as in $\Psi_{\text{sc}}^{1, 0}$, leaving an additional vanishing factor $r^{-1}$). We see then that 
$$
(D_r + \lambda) r^{-1} A_+ u \in H^{*, \sw+1; 0, \smo-1}_+
$$
microlocally on $\WF'(Q_3)$. Now using the ellipticity of $D_r + \lambda$ on this set, we find that 
$r^{-1} A_+ u \in H^{*, \sw+1; 0, \smo-1}_+$, which is equivalent to $A_+ u \in H^{*, \sw; 0, \smo-1}_+$ microlocally on $\WF'(Q_3)$. Together with the fact that we have already $Q_3 u \in H^{*, \sw; 0, \smo}_+$
shows that $Q_3 u \in H^{*, \sw; 1, \smo-1}_+$, with the required estimate. 

Now we show, by induction on $\kappa$, that we have the following estimate for all $(\kappa, \smo)$ provided that $\smo \geq 1$:
\begin{equation}
     \label{eq: low-reg mod-stronger}
     \norm[*,\sw; \kappa +1, \smo-1]{ Q_3 u } \le C \left(\norm[*,\sw; \kappa,\smo]{Q_3'u} + \| G_3 P u \|_{*, \sw +
     1; \kappa,\smo} + \norm[M,N]{u}\right).
  \end{equation}   
We have already shown this for $\kappa = 0$. So given $\kappa > 0$,
assume that \eqref{eq: low-reg mod-stronger} has already been proved
for all $(\kappa', \smo)$ with $\kappa' < \kappa$. The only thing left
to prove is to show that $A_+ u \in H^{*, \sw; \kappa, \smo - 1}_+$
microlocally on $\WF'(Q_3)$, with the corresponding estimate. By
induction, this follows if we show that $A_+ u$ is in $H^{*, \sw;
  \kappa-1, \smo}_+$ microlocally on $\WF'(Q'_3)$ and $P A_+ u$ is in
$H^{*, \sw+1; \kappa-1, \smo}_+$ microlocally on $\WF'(G_3)$. The
first statement is immediate from the induction hypothesis. 
For the second, we commute $P$ and $A_+$, obtaining, for  certain constants $a, b$, 
$$
P A_+u = A_+ Pu  - 2 D_r r^{-1} A_+ u + r^{-1} (a D_r u + b u) - 2 r^{-2} \Omega u,
$$
where $\Omega$ involves only tangential differentiation of order at most two, with coefficients smooth on $M$. 
Since by assumption, $Pu \in H^{*, \sw +1; \kappa, \smo}_+$, microlocally on $\WF'(Q_3)$ (an assumption in force throughout this paragraph, but which we shall omit repeating), we have $A_+ Pu \in H^{*, \sw +1; \kappa-1, \smo}_+$. Since we already know that $u \in H^{*, \sw; \kappa, \smo}_+$, we see that the term $r^{-1} D_r A_+ u$ is in $H^{*, \sw +1; \kappa-1, \smo}_+$. Similarly the term $r^{-1} (a D_r u + b u)$  is in $H^{*, \sw+1; \kappa, \smo}_+$. 
Finally, we split the two tangential derivatives of $\Omega$ as above
to see that $r^{-2} \Omega u \in H^{*, \sw +1; \kappa-1, \smo}_+$. We
conclude that $P A_+ u$ is in $H^{*, \sw+1; \kappa-1, \smo}_+$ and
therefore $A_+ u$ is in $H^{*, \sw; \kappa, \smo-1}_+$, which combined with $u \in H^{*, \sw; \kappa, \smo}_+$ shows that $u \in H^{*, \sw; \kappa+1, \smo-1}_+$, with the corresponding estimate. The proof of \eqref{eq: low-reg mod-stronger} is complete, and immediately implies \eqref{eq: low-reg mod}.

We next turn to the proof of \eqref{eq: high-reg mod}. This works quite differently in relation to the two modules. At the incoming radial set $\mathcal{R}_-$, the module $\mathcal{M}_+$ is elliptic, while $\mathcal{N}$ is characteristic. The effect of the $\kappa$th power of the module $\mathcal{M}_+$ is thus just to increase the spatial order $\sw$ by $\kappa$. So, without loss of generality, we may assume that $\kappa=0$.

We then employ a very similar argument to the one above. Notice that,  instead of 
having $\sigma_{\text{base},0}(C_0) =  -\lambda (2 \sw + 1)$ at the radial set, as above, we now have 
$\sigma_{\text{base},0}(C_0)  =  \lambda (2 \sw + 1)$. On the other hand, now
$\sw > -1/2$, so the $2 \sw + 1$ factor has also switched sign, so
this symbol remains positive at the radial set (now
$\mathcal{R}_-$). Using the fact that the module $\mathcal{N}$ is
$P$-critical, we find that the matrix $C'$ in
this case is again positive definite at the incoming radial set. Then
we run the same argument as above, with the following twist: In
this case, the $F$ term arising from $[Q^* Q, P]$ has the same sign as
$C'$, namely it is positive, as it arises from minus the (rescaled)
Hamilton vector field $\mathsf{H}_p$ hitting $\sigma(Q)^2$. Taking
$\sigma(Q)$ to be a function only of $|\mu|_h$ at $x=0$ we see that
from \eqref{ac-Hvf} that $-\mathsf{H}_p (\sigma(Q)^2)$ is
nonnegative. Since this has the same sign as that of $C'$, to leading
order, we can discard this term up to a lower order term. This lower
order term accounts for the presence of $\| G_4 u \|_{*, \sw', \smo}$ in
the estimate; at first sight it appears that we could take $\sw'= \sw-
1/2$ (and then iterate to reduce $\sw'$ as much as we like), but the
regularization required to make the estimate hold requires that $\sw'$
is greater than the threshold value of $-1/2$. See the proof of
\cite[Proposition 5.26]{va18}, between (5.61) and (5.62), for the
details of the regularization step.

This proves the proposition for $H_+^{*, \sw; \kappa,\smo}$. For the space
$H_-^{*, \sw; \kappa,\smo}$, where the index $\kappa$ now indicates module regularity with respect to $\mathcal{M}_-$, a similar argument applies. Here, it is important that the module $\mathcal{N}$ is $P$-critical and not just $P$-positive.  In this case, the above threshold estimate is localized near $\mathcal{R}_+$.
For $\mathcal{M}_-$ regularity we proceed as in the proof of
$\mathcal{M}_+$ for $H_+^{*, \sw; \kappa,\smo}$, i.e.\ we use the ellipticity of
$\mathcal{M}_-$ elements at $\mathcal{R}_+$.  For $\mathcal{N}$
regularity, in the expression \eqref{eq:C'alphaalpha}
we have that $\sigma_{\text{base},0}(C_0) = - \lambda (2 \sw + 1)$ where
now $\sw > -1/2$ so that the symbol of $C_0$ is actually negative in
this case.  For $C'_{\alpha \alpha}$ to have a sign, we must then know
that the other terms defining it do are not too positive. The fact that $\mathcal{N}$ is $P$-critical means they
vanish, and hence the matrix $C'$ is negative definite, rather than positive definite, near $\mathcal{R}_+$.
Then all proceeds as above.
\end{proof}

\begin{cor}\label{cor:low} Suppose $-3/2 < \sw' <  \sw < -1/2$ and $\kappa \geq 1$. Then, under the same assumptions as in part (ii) of Proposition~\ref{radial point prop mod}, we have \eqref{eq: high-reg mod}. 
\end{cor}

\begin{proof}
This follows since the module $\mathcal{M}_+$ is elliptic at $\mathcal{R}_-$. So when $\kappa \geq 1$, the estimate \eqref{eq: high-reg mod} is equivalent to the same estimate with $\sw, \sw'$ increased by $\kappa$, and $\kappa$ set to zero. 
Under the assumption that $-3/2 < \sw' <  \sw < -1/2$ and $\kappa \geq 1$ this puts us in the range of applicability of part (ii) of Proposition~\ref{radial point prop mod}. 
\end{proof}

\subsection{Invertibility on module regularity spaces}
Our final piece of preparation for the proof of Theorem \ref{thm:Fred prop} is the following  result relating our module regularity spaces to variable order spaces.

\begin{lemma}\label{thm:containment-mod-var}
Assume $\mathsf{l}_+ \in C^\infty( \overline{{}^{\emph{sc}}T^*M})$ satisfies
\eqref{weight-radial} and \eqref{weight-monotone}.  Then
for $\epsilon < \delta$ and $\sw= -1/2 - \epsilon$,
  \begin{equation}
    \label{eq:14}
    \kappa\ge 1\ 
 \Rightarrow  H_{+}^{s,\sw; \kappa, \smo} \subset H^{s, \mathsf{l}_+}.
  \end{equation}
\end{lemma}
\begin{proof}
  Since $H^{s,\sw;\kappa, \smo}_+ \subset H_+^{s,\sw;1,0} $, to show
  \eqref{eq:14}, it suffices to assume $\kappa = 1$ and $k = 0$. Let $U$ be a small neighborhood of $\mathcal{R}_+$ near which $\mathsf{l}_+ = -1/2 - \delta$, and let $V \Subset U$ be a smaller neighborhood of $\mathcal{R}_+$. For each $\q \in  \overline{{}^{\text{sc}}T^*M} \setminus V$, there is an element $A_\q$ of $\mathcal{M}_+$, elliptic on a neighborhood $U_\q$ of $\q$. Form a partition of unity subordinate to the cover of $\overline{{}^{\text{sc}}T^*M}$, consisting of $U$ and finitely many of the $U_\q$, say, $U_{\q_1}, \dots U_{\q_m}$. We take $\{Q_{\q_j}\}_{j=1} ^m, \, Q \in \Psi_{\text{sc}}^{0,0}$ to be the corresponding left quantizations of the microlocal cutoffs that comprise the partition of unity. Clearly, $Qu \in
  H^{s,\sw}$, since $u \in H^{s,\sw}$. Thus $Qu \in H^{s, \mathsf{l}_+}$ since, on $\WF'(Q)$, we have $\mathsf{l}_+ = -1/2 - \delta <
  -1/2 - \epsilon = \sw$.  
  
  On the other hand, by the assumption of module regularity, each $A_{\q_j} u \in H^{s,\sw}$. Because $A_{\q_j}$ is order $(1,1)$ and $\WF'(Q_{\q_j}) \subset \Ell (A_{q_j})$, microlocal elliptic regularity (Proposition~\ref{elliptic reg})  asserts that  $Q_{\q_j}u \in H^{s+1, \sw+ 1}$. Then we note that $H^{s+1, \sw+ 1}  \subset H^{s, \mathsf{l}_+}$ since $\sw +1 = 1/2 - \epsilon \ge -1/2 + \delta = \max \mathsf{l}_+$ for $\delta$ sufficiently small.
\end{proof}

\begin{rem} The point of this Lemma is that the different behaviour at the incoming and outgoing radial set is enforced by the module regularity instead of by a variable weight function, due to the fact that the module $\mathcal{M}_+$ is elliptic at the incoming radial set $\mathcal{R}_-$ but characteristic at the ougoing radial set $\mathcal{R}_+$. 
\end{rem}

We are now in a position to prove Theorem \ref{thm:Fred prop}. 

\begin{proof}[Proof of Theorem \ref{thm:Fred prop}] Let $s \in \RR$, $\sw \in (-3/2, -1/2)$, $\kappa \geq 1$ and $k$ be given. 
We first combine the estimates \eqref{eq: elliptic reg mod}, \eqref{eq: standard-prop mod}, \eqref{eq: low-reg mod} and \eqref{eq: high-reg mod} (in the latter case, for $\sw \in (-3/2, -1/2)$ as allowed by Corollary~\ref{cor:low}). This is done exactly as in Section~\ref{subsec:global}, so we do not repeat the argument.  We obtain, for $u \in H^{M,N}$ such that $Pu \in H^{s,\sw; \kappa,\smo}_+$, that $u \in H^{s,\sw; \kappa,\smo}_+$ and 
\begin{equation}\label{global-module-reg-est}
\| u \|_{s, \sw; \kappa,\smo} \leq C \Big(  \| Pu\|_{s-2, \sw; \kappa,\smo} + \| u\|_{M,N} \Big) .
\end{equation}
We next use Lemma~\ref{thm:containment-mod-var} to assert that $H_+^{s, \sw; \kappa, \smo} \subset
H^{s, \sfr_+}$, provided $\sw$ is sufficiently close to $-1/2$. The proof Lemma \ref{thm:containment-mod-var} may be trivially modified to also show  $H_+^{s-2, \ell ; \kappa, k} \subset H^{s-2, \sfr_+ +1 }$.
Thus we have the following diagram: 
\begin{equation}\label{eq:Setup}
	\xymatrix{ \mathcal{X}^{s, \sfr_+}\ar[r]^{P} & \mathcal{Y}^{s-2, \sfr_+ + 1} \\
	 \mathcal{X}^{s, \sw; \kappa, \smo}_+   \ar@{^(->}[u]  &
         \mathcal{Y}^{s-2, \sw + 1; \kappa , \smo}_+  \ar@{^(->}[u]  }.
\end{equation}
Our goal is to show that the restriction of $P$ to $\mathcal{X}^{s, \sw; \kappa, \smo}$ yields a bijection  $\mathcal{X}^{s, \sw; \kappa, \smo} \to \mathcal{Y}^{s-2, \sw + 1; \kappa , \smo}_+$. Injectivity follows immediately since the top row is an injective map. To show surjectivity, we suppose $f \in  \mathcal{Y}^{s-2, \sw+ 1; \kappa, 
  \smo}_+  =  H^{s-2, \sw + 1; \kappa , \smo}_+$. In, particular $f \in
H^{s - 2, \sfr_+ + 1}$. So, by surjectivity of the top row there is
$u \in \mathcal{X}^{s, \sfr_+}$ with $Pu = f$.  Then, 
thanks to \eqref{global-module-reg-est}, we see that, provided $(M,N)$
are sufficiently small, $u \in H^{s, \sw; \kappa,\smo}_+$. As a
bounded bijection, the map $\mathcal{X}^{s, \sw; \kappa, \smo}_+ \to
\mathcal{Y}^{s-2, \sw + 1; \kappa , \smo}_+$ is automatically a
Hilbert space isomorphism. \end{proof}

\section{Proof of the Main Theorem}\label{sec:proof}
In this section, we find a nonlinear eigenfunction with prescribed incoming data by finding a fixed point of the map in \eqref{Phi}, which we shall show is a contraction map on the space $\mathcal{X}_+^{2, \sw; 1, \smo}$. 
Here and below we fix $s=2$, and let $\smo$ be any integer strictly larger than $(n-1)/2$. We also put $\sw = -1/2 - \delta$, for some fixed $\delta$ with $0 < \delta \leq (4p)^{-1} \leq 1/8$.

\subsection{Linear eigenfunction}

Fix $f \in H^{\smo+2}(\p M)$.  Let $u_0$ be the unique solution to the free equation
\begin{equation*}
P(\lambda)u_0= 0,
\end{equation*}
subject to the condition that the coefficient on its incoming part at
infinity is $f$.  By \cite{RBMSpec}, if $f$ is $C^\infty$,  we have a decomposition
\begin{equation}
u_0 = r^{-(n-1)/2} (e^{- i \lambda  r} g_- + e^{i \lambda r}
g_+),\label{eq:u_f_decomp}
\end{equation}
where the $g_\pm \in C^\infty(M)$, and $g_- |_{\partial M} = f$. For
$f$ of finite regularity, this expansion only makes sense in a
distributional sense. In our case, we only use the `leading part' of
this expansion to express the linear eigenfunction $u_0$ as an element
in $H_-^{2, \sw;\mathsf{} 1, \smo}\oplus H_+^{2, \sw; 1, \smo}$. 
Thus, let 
$$
u_-(r, y)  = \chi(r) r^{-(n-1)/2} e^{- i \lambda  r} f(y) , \quad u_+ := u_0 - u_-. 
$$
Here $\chi$ is a cutoff function, supported in $r > R$ and identically equal to $1$ for $r \geq 2R$.
By inspection, we see that $u_- \in H_-^{2,\sw; 1, \smo} \cap H_-^{0,\sw; 1, \smo+2}$. Moreover, it is clear that if $\| f \|_{H^{\smo+2}(\partial M)}$ is sufficiently small, then $u_-$ is small in the norms of both these spaces. 

Moreover, by direct calculation
we have 
$$
(Pu_-) (r, y) = \tilde \chi(r) r^{-(n+3)/2} e^{- i \lambda  r} g(r, y), \quad x = r^{-1}
$$
where $\tilde\chi$ is a similar cutoff function, equal to $1$ on $\supp \chi$ and supported in $r > R$, and $g(r, y)$ is a smooth function of $r^{-1}$ with values in $H^k(\partial M)$. The key point is the gain of two powers of $r^{-1}$ as $r \to \infty$.  It follows that 
$$
Pu_- \in H_-^{0, \sw+2; 1, \smo} . 
$$
Now we want to view this as an element of $H_+^{0, \ell + 1; 1, \smo}$; to accommodate the $\mathcal{M}_+$-module regularity of order $k=1$, we lose one order of vanishing. Thus 
$$
Pu_- \in H_+^{0, \sw + 1; 1, \smo} = \mathcal{Y}_+^{0, \sw + 1; 1, \smo}.
$$
Then we claim that $u_+$ is equal to $- R(\lambda + i0) ( Pu_- )$. 
Indeed, $v := u_- - R(\lambda + i0) \big( Pu_- \big)$ solves 
$$
Pv = 0
$$
and according to Theorem~\ref{thm:Fred prop},  $R(\lambda + i0) \big( Pu_- \big)$ is in $\mathcal{X}_+^{2, \sw; 1, \smo}$; in particular, it has no incoming data at order $r^{-(n-1)/2}$ due to the module regularity at the incoming radial set ($\mathcal{M}_+$ is elliptic at this set). Thus, $v$ is the linear eigenfunction with incoming data equal to $f$, so it coincides with $u_0$ by definition. Hence $u_0 - u_- =  v - u_- = - R(\lambda + i0) ( Pu_- )$.


\subsection{Contraction mapping on $\mathcal{X}_+^{2, \sw; 1, \smo}$}

We return to the discussion of Section~\ref{subsec:strategy}. There, it was explained how finding a nonlinear eigenfunction amounts to finding a fixed point of the map $\Phi$ given by 
\begin{equation}\label{Phi2}
\Phi(w) = u_+ + R(\lambda + i0) \big( N[u_- + w] \big). 
\end{equation}

Let us check that a fixed point $w$ provides us with a nonlinear eigenfunction $u := u_- + w$. Adding $u_-$ to both sides of \eqref{Phi2}, we obtain 
\begin{equation}\label{Phi3}
\Phi(w) + u_- = w + u_- = u_+ + u_- + R(\lambda + i0) \big( N[u_- + w] \big) = u_0 + R(\lambda + i0) \big( N[u_- + w] \big).
\end{equation}
Thus, 
\begin{equation}\label{Phi4}
u =  u_0 + R(\lambda + i0) \big( N[u] \big).
\end{equation}
Now we apply $P$ to both sides. This annihilates $u_0$ and we find that 
$
Pu = N[u],
$
as claimed.

We now show that $\Phi$ is a contraction mapping on $\mathcal{X}_+^{2, \sw; 1, \smo}$, provided that $\| f \|_{H^{\smo+2}(\partial M)}$ is sufficiently small (and hence $\| u_- \|_{H_+^{2, \sw; 1, \smo}}$ is small), and provided that $w$ is small. The first thing to check is that $\Phi$ \emph{is} a mapping on this space. 

We have already seen that $u_+$ lies in this space, since $u_+$ is in $H_+^{2, \sw; 1, \smo}$ and $Pu_+ = - P u_- $ is in the space $H_+^{0, \sw + 1; 1, \smo}$ from the discussion above. 

Next, recall that the nonlinear term $N[v]$ is a product of  $\tilde p \geq p$ factors of the form $Qv$ or  $Q\overline{v}$, where $Q$ is a scattering differential operator of order $(2,0)$ (in the case of $\RR^n$ it just means a combination of the usual coordinate partial derivatives multiplied by $C^\infty(M)$ functions; see Remark~\ref{rem:clarify}). This is, therefore, a product of $\tilde p$ factors, each of which lies in $H_+^{0, \sw; 1, \smo}$. It follows that $N[u_- + w]$ is a finite sum of products of such factors. We have already seen that
$u_-$ lies in $H_-^{2, \sw; 1, \smo}$, and $w$ by assumption lies in $H_+^{2, \sw; 1, \smo}$. Also, we notice that complex conjugation is an involution between $H_-^{0, \sw; 1, \smo}$ and $H_+^{0, \sw; 1, \smo}$.
So $N[u_-+w]$ is a sum of products of factors, each of which lies in $H_-^{0, \sw; 1, \smo}$ or $H_+^{0, \sw; 1, \smo}$. 
Applying Corollary~\ref{cor:prod}, we find that the product lies in $H_+^{0, \sw'; 1, \smo}$, provided that for all $\tilde p \geq p$,  
\begin{equation}\label{rr'}
\sw' \leq \tilde p\sw + \frac{(\tilde p-1)n}{2} - 1. 
\end{equation}
We would like to know when this product is in $H^{0, \sw + 1; 1, \smo}$. This is the case provided that 
\begin{equation}\label{weight-ineq}
\sw + 1  \leq \tilde p\sw + \frac{(\tilde p-1)n}{2} - 1 \quad \Longleftrightarrow \quad 2 \leq (\tilde p-1) \sw  + \frac{(\tilde p-1)n}{2}.
\end{equation}
Since $\sw \geq -5/8$, the RHS is increasing in $\tilde p$. So it is only necessary to demand \eqref{weight-ineq} for $\tilde p = p$. Since $\sw < -1/2$, a necessary condition is that 
\begin{equation}\label{pn-req}
2 < (p-1) \frac{n-1}{2}, 
\end{equation}
which is precisely condition \eqref{pn}. When this holds, we automatically have 
\begin{equation}
5/2 \leq (p-1) \frac{n-1}{2}  \label{eq:999}
\end{equation}
since $n$ and $p$ are integers. It is straightforward to check that provided $0 < \delta \leq (4p)^{-1}$, given \eqref{eq:999}, we have \eqref{weight-ineq}, and in fact, in anticipation of Proposition~\ref{prop:expansions}, we note that  we can take $\sw' = 3/4$ in \eqref{rr'}.

Next, we verify that $\Phi$ is a contraction on a small ball in $\mathcal{X}_+^{2, \sw; 1, \smo}$, provided that the prescribed incoming data $f$ is small in $H^{\smo+2}(\partial M)$. We have
\begin{equation}\label{Ndiff}
\Phi(w_1) - \Phi(w_2) = R(\lambda + i0) \big( N[u_- + w_1] - N[u_- + w_2] \big).
\end{equation}
Since $N$ is a monomial of degree $p$, the RHS is a sum of terms the form  $Q(w_1 - w_2)$ or its complex conjugate, times a monomial of degree $p-1$ in various $Q'u_-$, $Q''w_1$ or $Q'''w_2$ or their complex conjugates, where the $Q$, $Q'$, etc., are scattering differential operators of order $(2,0)$.

Let $\eta > 0$ be a small parameter, to be chosen later. By direct calculation, we see that the map 
\begin{equation}\label{fu-}
f \mapsto u_-, \quad u_- = \chi(r) r^{-(n-1)/2} e^{- i \lambda  r} f(y),
\end{equation}
is a bounded map from $H^{\smo+2}(\partial M)$ to $H_-^{2, \sw; 1, \smo}$, so we may assume that $u_-$ is sufficiently small in this norm, say $\leq \eta$. Supposing that $w_1$ and both $w_2$ are both less than  $\eta$ in the norm  $H_+^{2, \sw; 1, \smo}$, then we find that $N[u_- + w_1] - N[u_- + w_2]$ is a finite number, say $c(p)$, of terms each of which is in $H_+^{0, \sw+1; 1, \smo}$ by Corollary~\ref{cor:prod} with norm in this space bounded by 
$$
C \| w_1 - w_2 \|_{H_+^{2, \sw; 1, \smo}} \Big(  \| u_- \|_{H_-^{2, \sw; 1, \smo}} +  \| w_1 \|_{H_+^{2, \sw; 1, \smo}}  +  \| w_2 \|_{H_+^{2, \sw; 1, \smo}} \Big)^{p-1} . 
$$
Applying $R(\lambda + i0)$, the inverse of $P$ acting between $\mathcal{Y}_+^{0, \sw + 1; 1, \smo}$ and $\mathcal{X}_+^{2, \sw ; 1, \smo}$, the norm of $\Phi(w_1) - \Phi(w_2)$ in $\mathcal{X}_+^{2, \sw ; 1, \smo}$ is at most $\| w_1 - w_2 \|_{H_+^{2, \sw; 1, \smo}}$ times $c(p) C\| R(\lambda + i0) \| (3\eta)^{p-1}$. 

It follows that provided $\eta$ is chosen small enough so that $c(p) C\| R(\lambda + i0) \| (3\eta)^{p-1}$ is strictly less than $1$, the map $\Phi$ is a contraction on the ball of radius $\eta$ centred at the origin in $\mathcal{X}_+^{2, \sw ; 1, \smo}$. By the contraction mapping theorem, we deduce the existence of a fixed point $w \in \mathcal{X}_+^{2, \sw ; 1, \smo}$. In view of the previous discussion this furnishes us with a nonlinear eigenfunction $u_- + w$. 


        
\subsection{Outgoing boundary data}
Continuing the proof of Theorem~\ref{thm:main2}, we show that $w$, the fixed point of $\Phi$ given incoming data $f$, has zero incoming boundary data and well-defined outgoing boundary data. 

\begin{prop}\label{prop:expansions}
Let $u = u_- + w$ be the nonlinear eigenfunction constructed above given $f \in H^{\smo+2}(\partial M)$. Then $u$ has an asymptotic expansion at infinity of the form
\begin{equation}
u = r^{-(n-1)/2} \Big( e^{-i\lambda r} f(y) + e^{i\lambda r} b(y) + O(r^{-\epsilon'}) \Big), \quad r \to \infty, 
\end{equation}
for some $\epsilon' > 0$, where $b \in H^{k}(\partial M)$.  
\end{prop}

\begin{proof}
We know that $Pu = N[u]$, and, in view of the discussion below \eqref{pn-req}, that  
the RHS is in $H^{0, 3/4; 1,\smo}_+$. 

The proof is therefore completed by the following lemma. 
\end{proof}

\begin{lemma}
Suppose  $\sw \in (-3/2, -1/2)$, $\smo > (n-1)/2$,  and that $w \in H^{2, \sw; 1,\smo +2}_+$ satisfies the equation 
\begin{equation}\label{ueqn}
P w = F, \quad F \in H_+^{0
, 1/2 + \epsilon; 1, \smo}(M)
\end{equation}
for some $\epsilon > 0$. 
Then $\lim_{r \to \infty} r^{(n-1)/2} e^{-i\lambda r} w(r, \cdot)$ exists in $H^{\smo}(\partial M)$. Letting $b  \in H^k(\partial M)$ denote the limit, we have 
\begin{equation}\label{w-asympt}
r^{(n-1)/2} e^{-i\lambda r} w(r,\cdot) - b  =  O(r^{-\epsilon'}) \text{ in } H^{\smo}(\partial M), \quad r \to \infty
\end{equation}
for any $\epsilon' < \epsilon$. 
\end{lemma}

\begin{rem} Since $H^{\smo}(\partial M)$ embeds into $C(\partial M)$, due to the assumption $\smo > (n-1)/2$, this also shows that we have the asymptotic \eqref{w-asympt} pointwise in $y \in \partial M$. \end{rem}

\begin{proof}
It suffices to decompose $w = w_+ + w_-$, where 
$$
r^{(n-1)/2} e^{-i\lambda r} w_+(r,\cdot) -  b =  O(r^{-\epsilon'}) \text{ in } H^{\smo}(\partial M), \quad r \to \infty
$$
and 
$$
r^{(n-1)/2} e^{+i\lambda r} w_-(r, \cdot) =  O(r^{-\epsilon'}) \text{ in } H^{\smo}(\partial M), \quad r \to \infty. 
$$
We do this by choosing a pseudodifferential cutoff $B \in
\Psi_{\text{sc}}^{0,0}$ such that $B$ is microlocally equal to the identity
near a neighbourhood $U$ of the outgoing radial set $\mathcal{R}_+$, and microlocally equal
to zero outside some slightly larger neighbourhood  $V$, i.e., for some
open neighborhoods $U \subset V$ of $\mathcal{R}_+$ we have 
$\WF'(I - B) \cap U = \varnothing, \WF'(B) \subset V$. Then we set 
\begin{equation}
w_+ = Bw, \quad w_- = (\Id - B) w.
\end{equation}
From \eqref{ueqn} we get 
\begin{equation}
P w_+ = P (Bw) = B F + [P,
B]w.\label{eq:29}
\end{equation}
We claim that the RHS is in $H_+^{0, 1/2 + \epsilon; 0, \smo}$. Certainly
this is true for the term $BF$ since $F$ is in this space 
and $B \in \Psi_{\text{sc}}^{0,0}$. For the
term $[P, B]w$, we claim that
\begin{equation}
[P, B] = r^{-2} A \mbox{ where } A \in
\mathcal{M}_+.\label{eq:28}
\end{equation}
Since $w \in H_+^{2, \sw; 1, \smo+2}$, this would imply $A w  \in H_+^{2, \sw; 0, \smo +2}$, and so $[P, B]w = r^{-2} A w \in H_+^{2, \sw + 2; 0, \smo}
\subset H_+^{0, 1/2 + \epsilon; 0, \smo}$ for  $\epsilon$ sufficiently small.  But \eqref{eq:28} follows immediately for $V$
sufficiently small since $r^2 [P, B]$ has order $(1, 1)$, and  is
microsupported in $V \setminus U$. Therefore it is characteristic at $\mathcal{R}_+$, which is a sufficient condition for an operator of order $(1,1)$ to belong to $\mathcal{M}_+$.

Write $\wtilde_+ = \chi(r) r^{(n-1)/2} e^{-i\lambda r} w_+$, where $\chi$ is supported in $r > R$ and identically $1$ near $r \geq 2R$, and we assume $R$ is sufficiently large that the support of $\chi$ is contained in a collar neighbourhood of the boundary.
Our first goal is to show that $\wtilde_+(r,y)$ has a limit $b(y)$ as $r \to \infty$, and that $\wtilde_+(r,y) - b(y) = O(r^{-\epsilon'})$. 
To do this, we write the operator $P$ in the form \eqref{P-adapted}: 
\begin{equation}\label{ueqn2}
P =  D_r^2 - \lambda^2 - \frac{i (n-1)}{r} D_r + r^{-2}
Q + r^{-2} \tilde Q, 
\end{equation}
where $Q$ is a second order differential operator involving only tangential $D_{y_j}$ derivatives, and $\tilde Q$ is a scattering differential operator of order $(1,0)$. We then apply this decomposition to \eqref{eq:29} and rearrange to obtain 
\begin{multline}\label{ueqn3}
\Big( D_r + \lambda \Big) \Big( D_r -  \lambda  - \frac{i (n-1)}{2r} \Big) w_+ \\ = BF + [P, B]w + \frac{i (n-1)}{2r^2} (r(D_r - \lambda)) w_+ + \Big(r^{-2} Q  + r^{-2} \tilde Q \Big) w_+  .
\end{multline}
Notice that the RHS is in $H_+^{0, 1/2 + \epsilon; 0, \smo}$, using assumption \eqref{ueqn} for $F$, the $\mathcal{M}_+$ module regularity of order $1$ for the term involving $r(D_r - \lambda) \in \mathcal{M}_+$, and the $\mathcal{N}$ module regularity of order $\geq 2$ for the tangential derivatives. 
Moreover, the operator $D_r + \lambda$ is elliptic everywhere except at the set $\{ \nu = -\lambda \}$; in particular, it is elliptic on $\WF'(B)$, provided that $V$ is taken sufficiently small.  Thus we may invert this operator microlocally, obtaining 
\begin{equation}\label{ueqn4}
 \Big(D_r -  \lambda - \frac{i(n-1)}{2r}  \Big) w_+  \in H_+^{0, 1/2 + \epsilon; 0, \smo}.
\end{equation}
Now, observing that 
$$
D_r \wtilde_+ = r^{(n-1)/2} e^{-i\lambda r} \Big( D_r -  \lambda -  \frac{i(n-1)}{2r}  \Big) w_+ + (D_r \chi) r^{(n-1)/2} e^{-i\lambda r}  w_+,
$$
we find that 
$$
D_r \wtilde_+ \in H_+^{0, 1/2 + \epsilon - (n-1)/2; 0, \smo} \Longleftrightarrow D_r \wtilde_+ \in  r^{n/2 - 1 - \epsilon} L^2\big( (([R, \infty), r^{n-1} dr); H^{\smo}(\partial M) \big)
$$
where we used the support property of $D_r \chi$ for the inclusion in $H_+^{0, n/2 -1 + \epsilon; 0, \smo}$ of the second term.  We can express this with respect to the measure $dr$ as follows: 
$$
D_r \wtilde_+ \in  r^{-1/2 - \epsilon} L^2\big( ([R, \infty),  dr); H^{\smo}(\partial M) \big) \subset r^{-\epsilon'} 
L^1\big( ([R, \infty),  dr); H^{\smo}(\partial M) \Big).
$$
Notice that, by assumption, $w$ is locally $H^1$ in $r$ with values in $H^k(\partial M)$, so it is therefore in $H^k(\partial M)$ for each fixed $r$. We can therefore integrate to infinity and find that 
$$
b(y) = \wtilde_+(R, y) + \int_R^\infty \partial_r \wtilde_+(r', y) \, dr'
$$
is well-defined as an element of $H^{\smo}(\partial M)$. 
Moreover, 
$$
\wtilde_+(r, y) - b(y) = -\int_r^\infty \frac{d}{dr'}  \wtilde_+(r', y) \, dr' = O_{H^{\smo}(\partial M)}(r^{-\epsilon'}). 
$$

A very similar argument can be applied to the $w_-$ term. We define 
$\wtilde_- = r^{(n-1)/2} e^{i\lambda r} w_-$ and compute as above. 
However, we switch the sign of $\lambda$ in \eqref{ueqn3} to obtain
\begin{multline}\label{ueqn5}
\Big( D_r - \lambda \Big) \Big( D_r +  \lambda  - \frac{i (n-1)}{2r} \Big) w_- \\ = (\Id - B)F - [P, B]w + \frac{i (n-1)}{2r} (D_r + \lambda) w_- + \Big(-r^{-2} \Delta_{\partial M} + r^{-1}Q'  \Big) w_-  .
\end{multline}
where $Q' = Q_1' + r^{-1}Q_2'$, $Q_i' \in \Psi_{\text{sc}}^{2, 0}, i =
1,2$ with $Q_1'$ a scattering differential operator involving only
tangential derivatives.
On the microlocal support of $\Id - B$, the module $\mathcal{M}_+$ is elliptic, hence $w_-$ is actually in $H^{2, \sw + 1; 0, \smo+2}$ in this region. So the third term on the RHS is in $H_+^{1, \sw +2 ; 0, \smo+2}$.
We see that the RHS is in $H_+^{0, 1/2 + \epsilon; 0, \smo}$ as before. 
We may assume that $D_r -  \lambda$ is elliptic on the microsupport of $\Id - B$, so we may invert it microlocally to obtain 
\begin{equation}\label{ueqn6}
\Big( D_r +  \lambda  - \frac{i (n-1)}{2r} \Big) w_- \in H_+^{0, 1/2 + \epsilon; 0, \smo}. 
\end{equation}
Now, observing that 
$$
D_r \wtilde_- = r^{(n-1)/2} e^{i\lambda r} \Big( D_r + \lambda -  \frac{i(n-1)}{2r}  \Big) w_+ + (D_r \chi) r^{(n-1)/2} e^{i\lambda r}  w_+,
$$
we find that 
$$
D_r \wtilde_- \in r^{-1/2 - \epsilon} L^2\big( (([R, \infty),  dr); H^{\smo}(\partial M) \big) \subset r^{-\epsilon'} 
L^1\big( (([R, \infty),  dr); H^{\smo}(\partial M).
$$
The rest of the argument can be followed to obtain a limit 
$$
b_-(y) = \lim_{r \to \infty} \wtilde_-(r, y) 
$$
in $H^{\smo}(\partial M)$, with
$$
\wtilde_-(r, y) - b_-( y) =  O_{H^{\smo}(\partial M)}(r^{-\epsilon'}).  
$$
This is only possible if $b_-$ vanishes identically. Indeed, otherwise
  $w_-$ would fail to be in  $H^{s, -1/2}$; however, since  $\WF'(\Id - B)
  \cap \mathcal{R}_+ = \varnothing$ and the module 
$\mathcal{M}_+$ is  elliptic off $\mathcal{R}_+$, $(\Id - B) w
\in H^{s, \sw + 1}$.  This completes the proof of the lemma. \end{proof}


\subsection{Uniqueness} 
To complete the proof of Theorem~\ref{thm:main2}, we show that the solution $u$ obtained above is unique in the following sense:

\begin{prop}\label{thm:uniqueness} Suppose that \eqref{pn} is satisfied, that $N$ satisfies the conditions of Theorem~\ref{thm:main2}, and that $f$ is sufficiently small in $H^{\smo+2}(\partial M)$, so that the proof above of the existence of a nonlinear eigenfunction $u$ with incoming data $f$  is valid. Let $u_-$ be given in terms of $f$ by \eqref{fu-}. 

Then the solution $u$ is unique in the following sense. 
  Let $u_1, u_2$ satisfy $P u_i = N[u_i]$
  and assume $u_i - u_- = w_i$ both lie in $H^{2,\sw;1,\smo}_+$.  Then there
  exists  $\eta > 0$ such that 
$$
\norm[H^{\smo+2}(\partial M)]{f}, \norm[H^{2, \sw; 1,\smo}_+]{w_1},
\norm[H^{2, \sw; 1,\smo}_+]{w_2} < \eta \implies u_1 = u_2.
$$
\end{prop}
\begin{proof}
Let $u_1$ and $u_2$ be nonlinear eigenfunctions as in the proposition. It suffices to show that the corresponding $w_i$ are both fixed points of the map $\Phi$, since a contraction map has only one fixed point. 

We first note that the $w_i$ are in $\mathcal{X}_+^{2, \sw; 1, \smo}$. First note that $u_- \in H_-^{2,  \sw; 1, \smo}$ and $w_i$ is by assumption in $H^{2,\sw;1,\smo}_+$, so, as we saw above, this means that $N[u_i] = N[u_- + w_i] \in H_+^{0, \sw +1; 1, \smo}$. Since $Pw_i = - Pu_- + N[u_i] \in H_+^{0, \sw +1; 1, \smo}$, this confirms that $w_i \in \mathcal{X}_+^{2, \sw; 1, \smo}$. 

Next, from 
$$
P(u_- + w_i) = N[u_i],
$$
we apply $R(\lambda + i0)$ and note that $R(\lambda + i0) Pw_i = w_i$ since $w_i \in \mathcal{X}_+^{2, \sw; 1, \smo}$, while, as we have seen, $R(\lambda + i0) P u_- = - u_+$. Therefore, 
$$
-u_+ + w_i = R(\lambda + i0) N[u_- + w_i]
$$
and this rearranges to $\Phi(w_i) = w_i$ for each $i$. The proof is complete. 
%
\end{proof}

\bibliographystyle{plain}
\bibliography{helmholtz}
\end{document}